\documentclass[a4paper,12pt]{amsart}

\usepackage{graphicx}
\usepackage{color,hyperref}
\usepackage{amssymb,amsfonts,amsmath,amstext,amsbsy,amsopn,amsthm,amscd}
\usepackage{tikz}
\usetikzlibrary{patterns} % Bibliothek für technische Schraffuren

\newtheorem{theorem}{Theorem}[section]
\newtheorem{lemma}[theorem]{Lemma}%[section]
\newtheorem{corollary}[theorem]{Corollary}%[section]
\newtheorem{proposition}[theorem]{Proposition}%[section]
\theoremstyle{definition}

\newtheorem{definition}[theorem]{Definition}%[section]
\newtheorem{assu}[theorem]{Assumption}%[section]
\newtheorem{rem}[theorem]{Remark}%[section]

\theoremstyle{remark}

\numberwithin{equation}{section}

\newcommand{\field}[1]{\mathbb{#1}}

\newcommand{\R}{\field{R}}

\DeclareMathOperator{\dist}{dist}  %distance
\DeclareMathOperator{\supp}{supp}  %support of a function
\DeclareMathOperator{\Int}{int}

\DeclareMathOperator*{\essinf2}{ess\, inf}

\newcommand{\x}{\mathrm{x}}  
\newcommand{\y}{\mathrm{y}} 
\newcommand{\z}{\mathrm{z}}  

\newcommand{\lozengeninety}{\rotatebox[origin=c]{90}{$\lozenge$}}
\newcommand{\lhdbig}{\scalebox{0.9}{$\lhd$}}

\begin{document}
\author{Joachim Rehberg}
\address{J. Rehberg, Weierstrass Institute, 
Anton-Wilhelm-Amo-Strasse 39, 10117 Berlin, Germany}
\email{joachim.rehberg@wias-berlin.de}

\author{Elmar Schrohe} 
\address{E. Schrohe. Leibniz University Hannover, Institute of Analysis, Welfengarten 1, 30167 Hannover, Germany }
\email{schrohe@math.uni-hannover.de} 

%\nofnmark{}				% keine Fu"snotenmarke

\title[Operators on Domains with Buried Boundary Parts]%
{Optimal Sobolev Regularity for Second Order Divergence Elliptic Operators on Domains with Buried Boundary Parts\\[2ex] {\tiny \it This article is dedicated to Prof. Hans Triebel\\ on the Occasion of his 90th Birthday}}
			
	% Math. Subject Classif.

\begin{abstract}
We study the regularity of solutions of elliptic second order boundary value problems on a bounded domain $\Omega$ in $\mathbb R^3$.
The coefficients may have certain discontinuities and the boundary conditions may be mixed, i.e. Dirichlet on one part $D$ of the boundary and Neumann on the complementing part. The peculiarity is that $D$ is partly `buried' in $\Omega$ in the sense that the topological interior of $\Omega \cup D$ properly contains $\Omega$.
The main result states that the regularity of the solution along the border of the buried contact behaves exactly
as the regularity for the solution of a mixed boundary value problem along the border between the Dirichlet 
and the Neumann boundary part.
\end{abstract}
\maketitle

\tableofcontents

\section{Introduction}

The most disruptive force in semiconductor devices is heat \cite{Ratho}, \cite{fischer}. It leads to the segregation of chemical compounds and eventually to the destruction of the device. 
In the mathematical theory of semiconductor modeling there exists up to now only one thermodynamically consistent model that includes the electron/hole transport combined with heat transfer \cite{Alb}. Unfortunately this system lacks parabolicity and therefore defies so far a rigorous mathematical analysis.
Our aim in this work is considerably more modest. We study a physical quantity such as the temperature or the electrostatic potential, subject to a corresponding elliptic equation. To fix ideas, consider the \emph{stationary} heat equation or
Poisson's equation for the electrostatic potential -- here as part of a semiconductor model, see  \cite{Gaj}.

The crucial feature is that the device contains a so called `buried contact' $B$ within a much less conducting
material: Think of a film of silver that lies inside the device but extends  to its boundary. Its thickness is assumed to be negligible compared to the other parameters of the device, so that it can be idealized as an interior {\em surface} with (possibly non-smooth) boundary.
In the stationary heat problem one may think of the film to be heated to a certain temperature from `outside' and in the 
semiconductor model  that a certain voltage is applied to the `outer part' of the film.

The question we are addressing here is the following: Given suitable data, which regularity can we expect near $B$ for the solutions to the above equations? 

\begin{figure}[h]
    \centerline{\includegraphics[scale=0.33]{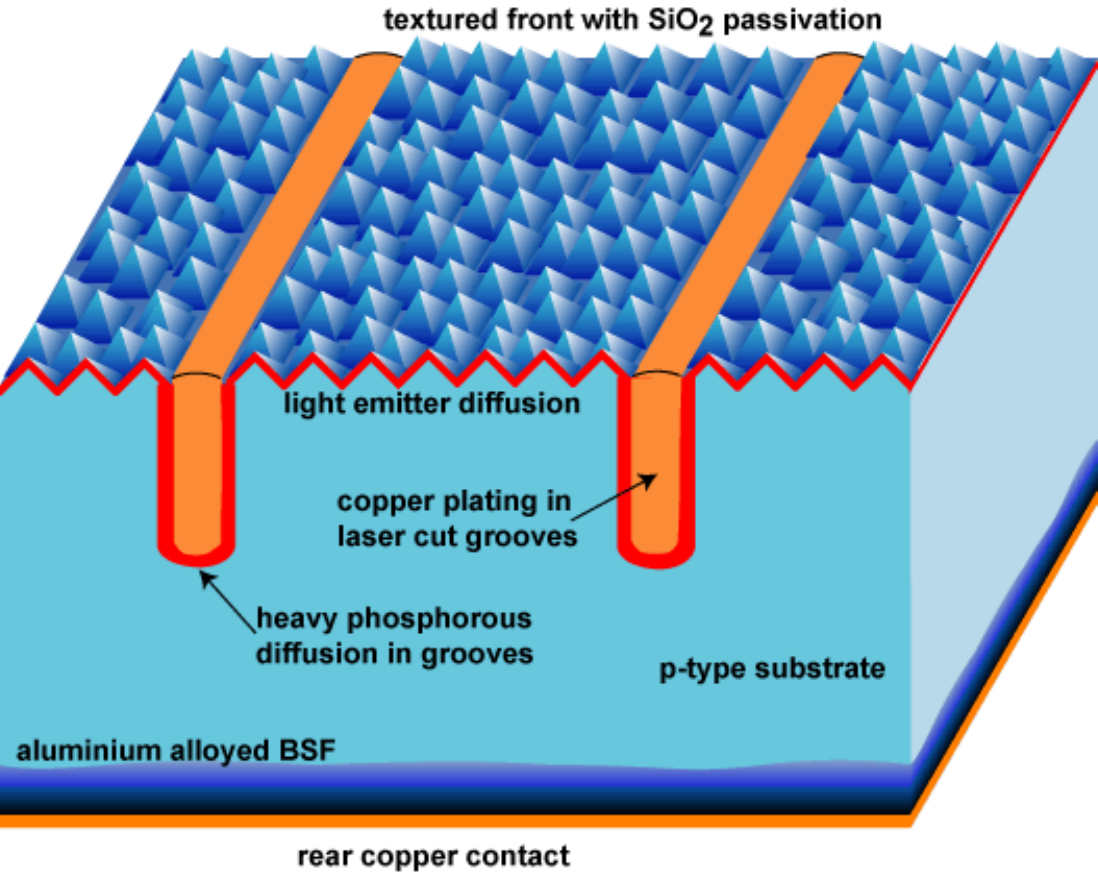}}
     \caption{Buried contact solar cell. Reproduced from \cite{Honsberg} with the kind permission of the authors}
  \end{figure}

For similar questions in semiconductor modeling, the threshold regularity is known to be $W^{1,3+\epsilon}$ as observed in the  pioneering paper \cite{Gaj}.
In \cite{dissreh} this was shown to lead to a satisfactory analysis of  the van Roosbroeck system 
even for the case where surface charge densities and avalanche generation are taken into account. 
Here, in the analysis of an evolution equation, it is necessary to identify the domain of the elliptic operator exactly in order to
treat the occurring nonlinearities suitably, see \cite{Pruess}.

In mathematical terms, the problem is the following: 
Let  $\Omega\subset \R^3$ be a bounded domain with boundary $\partial \Omega$. We assume that there exists a closed subset $D\subset\partial \Omega$ (the Dirichlet boundary) such that $\Omega$ is a proper subset of the interior $\widehat \Omega$ of $\Omega\cup D$, i.e., $\widehat \Omega\setminus \Omega\not=\emptyset$.   

Consider an elliptic equation
\begin{eqnarray*}
-\nabla \cdot \mu \nabla u &=&f \text{ in }\Omega\\
u&=&0\text{ on } D\\
\nu\cdot \mu \nabla u&=&0\text{ on } \partial \Omega \setminus D,
\end{eqnarray*}
where $\nu$ denotes the unit normal of the boundary.

We are mainly interested in the regularity of the solution $u$ near points in the `buried' part $D^\parallel= \widehat\Omega\setminus \Omega$ of $D$, 
since, outside $\overline {D^\parallel}$, one can apply the (local) regularity theory  in \cite{disser}.
It will be useful to decompose $\overline {D^\parallel}$ into four parts, later denoted by 
$D^\parallel_d$, $D^\parallel_c$, $R_d$ and $R_c$, see Definitions \ref{d-typoparall},  \ref{d-rigidincl0}, and \ref{d-typo2}. 
For a sketch see Figure \ref{fig:Model}, below. 

\begin{center}
\begin{figure}[htbp]
\centering
\begin{tikzpicture}[
    scale=1.4,
    % x-Achse um Faktor 0.5 verkürzt (Westsüdwest)
    x={(-0.554cm, -0.23cm)}, 
    y={(1cm, 0cm)}, 
    z={(0cm, 1cm)},
   mystyle/.style={line width=1.2pt}
]
    % Eckpunkte des Würfels [-1,1]^3
    \coordinate (A) at (-1, -1, -1); 
    \coordinate (B) at ( 1, -1, -1); 
    \coordinate (C) at ( 1,  1, -1); 
    \coordinate (D) at (-1,  1, -1); 
    
    \coordinate (E) at (-1, -1,  1); 
    \coordinate (F) at ( 1, -1,  1); 
    \coordinate (G) at ( 1,  1,  1); 
    \coordinate (H) at (-1,  1,  1); 

    % Eckpunkte der schraffierten Fläche {0} x [-1/2, 1/2] x [0, 1]
    \coordinate (S1) at (0, -0.5, 0);
    \coordinate (S2) at (0,  0.5, 0);
    \coordinate (S3) at (0,  0.5, 1);
    \coordinate (S4) at (0, -0.5, 1);

    % 1. Schraffierte innere Fläche ZUERST zeichnen (transparent ohne Hintergrundfüllung)
    \draw[pattern=north east lines, pattern color=blue!50!black, mystyle, draw=blue!60!black] 
        (S3) -- (S2) -- (S1) -- (S4); % -- cycle;

    % 2. Gestrichelte Kanten nach Vorgabe: (AB), (AD), (AE)
    \draw[dashed, gray, thick] (A) -- (B);
    \draw[dashed, gray, thick] (A) -- (D);
    \draw[dashed, gray, thick] (A) -- (E);

    % 3. Durchgezogene Kanten nach Vorgabe: (BC), (BF), (CG), (CD), (DH), (EF), (EH), (FG), (GH)
    \draw[thick] (B) -- (C);
    \draw[thick] (B) -- (F);
    \draw[thick] (C) -- (G);
    \draw[thick] (C) -- (D);
    \draw[thick] (D) -- (H);
    \draw[thick] (E) -- (F);
    \draw[thick] (E) -- (H);
    \draw[thick] (F) -- (G);
    \draw[thick] (G) -- (H);
    \draw[dashed,gray,very thick] (S3) -- (S4);

    % 4. Koordinatenachsen ab Ursprung (0,0,0)
    \draw[->, thick, red!70!black] (0,0,0) -- (1.5,0,0) node[anchor=north east] {$x$};
    \draw[->, thick, red!70!black] (0,0,0) -- (0,1.9,0) node[anchor=south west] {$y$};
    \draw[->, thick, red!70!black] (0,0,0) -- (0,0,1.6) node[anchor=south east] {$z$};

    % Ursprung markieren
    \fill[red!70!black] (0,0,0) circle (1pt);
    \node[red!70!black] at (0,0,0) [anchor=north] {$O$};

    % 5. Beschriftung der Eckpunkte des Würfels
%    \node[font=\footnotesize, below left] at (A) {$A$};
%    \node[font=\footnotesize, below right] at (B) {$B$};
%    \node[font=\footnotesize, above right] at (C) {$C$};
%    \node[font=\footnotesize, above left] at (D) {$D$};
%    \node[font=\footnotesize, left] at (E) {$E$};
%    \node[font=\footnotesize, right] at (F) {$F$};
%    \node[font=\footnotesize, above right] at (G) {$G$};
%    \node[font=\footnotesize, above left] at (H) {$H$};

    % Beschriftung der schraffierten Fläche (mit Führungslinie)
    \draw[thin, blue!60!black] (0, 0.2, 0.5) -- (-0.8, 0.25, 1.5) 
        node[right, align=left, font=\small] {$D^\parallel =\{0\} \times [-\frac{1}{2},\frac{1}{2}] \times [0,1[$};

    % Titel
 %   \node[above] at (0, 0, 1.6) {\textbf{Würfel $[-1,1]^3$ (mit verkürzter $x$-Achse)}};

\end{tikzpicture}
\caption{The cube $]-1,1[^3$ with the  surface $D^{\parallel}=\{0\}\times [-\frac12,\frac12]\times [0,1[$. The hatched interior is $D^\parallel_d$, the three thick boundary lines represent $D^\parallel_c$. 
The dashed line $\{0\}\times [-\frac12,\frac12]\times \{1\}$ on top is $R$, with $R_d =  \{0\}\times {]-\frac12,\frac12[}\times \{1\}$ and $R_c = \{(0,-\frac12,1[),(0,\frac12,1)\}$. 
}
\label{fig:Model}
\end{figure}
\end{center}

Our proof rests on four pillars: localization near these four parts, a transformation to suitable model constellations, 
freezing of coefficients  and the consideration of the model constellations. 
An important tool in the last step is the analysis of special  model sets treated in \cite{haller} by a  symmetrization/antisymmetrization procedure. 
%We are interested in the regularity of the solution $u$ near points in $\widehat\Omega\setminus \Omega$. 
%Our main theorem, resting on  results from \cite{haller},
%says that, under moderate assumptions on the geometry, the solution is again $W^{1,3+\epsilon}$ near these points. Our strategy is to localize the problem around the points under consideration and to reduce the localized problem by a $C^1$-transformation to one for 
%which the resulting geometry fits into a class of model constellations
%treated in \cite{haller} by a  symmetrization/antisymmetrization procedure.
Interestingly, the symmetrized part of the solution appears exactly as 
a solution of a \emph{mixed} boundary value problem. Therefore one can, on one hand, expect no better regularity
 than $W^{1,4-\epsilon}$ in view of Shamir's famous counterexample \cite{shamir}, but obtains, one the other hand, 
$W^{1,3+\epsilon}$ regularity for right hand sides in $W^{-1,3+\epsilon}$ for (possibly small) $\epsilon >0$.
So, generally speaking, our approach shows that the regularity of the solution at the border of the buried boundary part corresponds to the regularity occurring for a mixed boundary value problem at the border between Dirichlet and Neumann boundary part.

We conjecture that our regularity result also is of use for the investigation of `rigid inclusions' in mechanics as studied in \cite{khlud}.

%\section[Preliminaries ]{Preliminaries and general assumptions}\label{s-not}

\section[Notation, Assumptions,  Geometric Setting]
{Notation,  Assumptions, Geometric Setting}
\label{s-main3} 

\subsection{Notation} \label{ss-prelim}
We write variables in $\R^3$ in the form $\x=(x_1,x_2,x_3)$, $\y=(y_1,y_2,y_3)$, etc., 
and denote by $\mathrm e_1,\mathrm e_2, \mathrm e_3$ the unit vectors in the $x$, $y$ and $z$ direction, respectively. As usual, $B(\x;r)$ is the ball of radius $r$ around $\x$.

We will frequently identify a uniformly continuous function defined on a subset $V$ of $\R^3$ with its (unique) uniformly continuous extension to the closure $\overline V$. 
By $\|\cdot\|_X$ we denote the norm in the Banach space $X$. 
Finally, the letter $c$ denotes a generic constant, not always of the same value.

Moreover, we use the following notation: 
\begin{enumerate}
\item $\iota:\R^3 \to \R^3$ denotes the involutive  map given by reflection in the first coordinate $\iota(x_1,x_2,x_3)=(-x_1,x_2,x_3)$.
\item  ${\bf H}^\pm_j \subset \R^3$, $j=1,2,3$, are the half spaces $\{\y:  \pm y_j > 0\}$.
\item ${\bf  C}$  stands for the cube ${]-1,1}\,[^3$, and ${\bf Q}$ for its lower half, i.e., 
${\bf Q} = {\bf  C} \cap {\bf H}_3^-$. Moreover, 
${\bf C}_\pm = {\bf C} \cap {\bf  H}_1^\pm \text{ and } {\bf Q}_\pm = 
{\bf Q } \cap {\bf H}_1^\pm .$
\item A bijective map $\phi:V\to W$ between two subsets $V,W$ of $ \R^3$ is called {\em bi-Lipschitz},
if there exist positive constants $c_1$ and $c_2$ such that 
\begin{eqnarray}\label{bilip}
c_1|\x-\y|\le |\phi(\x) -\phi(\y) | \le c_2 |\x- \y|, \; \x, \y \in V.
\end{eqnarray}
It is not hard to see that $\phi$ then extends (uniquely) to
a bi-Lipschitz map $\widehat \phi: \overline V\to  \overline W$ between the closures of $V$ and $W$, satisfying \eqref{bilip} with the same constants $c_1$ and $c_2$. 

Recall that a Lipschitz function on a domain possesses in almost all points a classical (and hence a generalized) derivative, which is in norm not larger than the Lipschitz constant, cf. \cite[Sect. 3.1.2]{ev/gar}.

\item 
We call a bijective map $\phi:V\to W$  between open sets $V ,W\subset \R^3$ a $C^1$-diffeomorphism, if 
$\phi$ and $\phi^{-1}$ are continuously differentiable with \emph{uniformly continuous} derivatives. 
In this terminology, a $C^1$-diffeomorphism is, in particular,  a {bi-Lipschitz} mapping between $V$ and $W$. 
\item
Following \cite[Ch. 2.1]{wallin}, we call a closed set $E \subset \R^3$ a $2$-set if there are constants $c_\bullet,
c^\bullet$ such that 
\begin{equation} \label{e-twoset}
c_\bullet \, r^2 \le \mathcal H_2 \bigl (E \cap B(\x,r) \bigr ) \le c^\bullet \, r^2, \quad \x \in E, \, r \in {]0,1]},
\end{equation}
$\mathcal H_2$ being the two-dimensional Hausdorff measure on $\R^3$.
\end{enumerate} 

\subsection{General Assumptions}{\label{a-assugeneral}}

In the sequel, $\Omega \subseteq \mathbb R^3$ will denote the fixed bounded three dimensional domain we are interested in, and $D$  a closed subset of $\partial \Omega$, which has $3$-dimensional Lebesgue measure zero. We think of $D$ as the Dirichlet boundary of $\Omega$, and let $N=\partial \Omega\setminus D$.

The letter  $\Lambda$ stands for a bounded open set in $\R^3$ that will be used to study different constellations. Here the letter $E$ will be used for the closed part of $\partial \Lambda$, where Dirichlet conditions are imposed. 

The complements of $D$ in $\partial \Omega$ and of $E$ in $\partial\Lambda$, respectively, carry Neumann boundary conditions, see Remark \ref{r-neum}, below.

\begin{definition} \label{d-coeff}
A \emph{coefficient function} on an open subset $\Lambda$ of $\R^3$ is a bounded, measurable function $\rho$ 
on $\Lambda$, taking values in the set of \emph{real, symmetric} $3 \times 3$ matrices.  
If $\rho$ additionally satisfies the condition
\begin{equation} \label{e-ellip}
\essinf2_{{\x}\in \Lambda} \inf_{\|\xi\|_{\R^3}=1} \rho({\x})\xi\cdot \xi   >0,
\end{equation}
then it will be called elliptic. 
\end{definition}

While $\rho$ stands for a general coefficient function on $\Lambda$ we reserve the letter $\mu$ for the fixed coefficient function on $\Omega$ which we assume to be elliptic. 

\subsection{Boundary value problems}
For $q \in {]1,\infty [}$, $W^{1,q}(\Lambda)$  means the (complex-valued) Sobolev space on $\Lambda$. 
Given a closed subset $E $ of $\partial \Lambda$, we let
$W^{1,q}_E(\Lambda)$ be the closure of
\[ 
C^\infty_E(\Lambda):=\left \{v|_\Lambda: v \in C^\infty_0(\R^3),\, \supp v \cap E = \emptyset \right \}
\]
in $W^{1,q}(\Lambda)$. As usual, we write  $W^{1,q}_0(\Lambda)$ instead of $W^{1,q}_{\partial \Lambda}(\Lambda)$. Finally, 	
$W_E^{-1,q}(\Lambda)$ 
denotes the (anti)dual to 
$W^{1,q'}_E(\Lambda)$ with respect to an extension of the $L^2$ sesquilinear form, where $\frac {1}{q}+\frac {1}{q'}=1$.

Given a coefficient function $\rho$ on $\Lambda$, we define 
\[
-\nabla \cdot \rho \nabla:W^{1,2}_E(\Lambda) \to W^{-1,2}_E(\Lambda) 
\]
by
\[\langle -\nabla \cdot \rho \nabla v,w\rangle =\int_\Lambda \rho \nabla v\cdot \nabla 
\overline w \,d\x, \quad v, w \in W^{1,2}_E(\Lambda),
\]%end{equation}
$\langle \cdot , \cdot \rangle$ denoting the sesquilinear pairing between $W^{1,2}_E(\Lambda)$
and $W^{-1,2}_E(\Lambda)$. 

\begin{rem} \label{r-neum} It is well-known that the part of this operator in $L^2(\Lambda)$ leads to a homogeneous Dirichlet condition (in the sense of traces)
on $E$ and a (generalized) homogeneous Neumann condition $\nu \cdot \rho \nabla \psi =0$
(for  $\psi$ in the domain)  
on $N:=\partial \Lambda \setminus E$, see
\cite[Section 1.2]{cia}  or  \cite[Chapter II.2]{ggz}, compare also \cite{daners}.
\end{rem}

\subsection{Localization and a transformation theorem}\label{localization}
We start by  recalling Gr\"oger's localization principle, \cite[Lemma 2]{groeger89}, for elliptic second order 
operators in the $W^{1,q}$ scale. It shows that local regularity implies global regularity. 

\begin{theorem} \label{t-groeger_local_global} 
Let $\Lambda \subset \R^3$ be a domain and $E \subset \partial \Lambda$ a closed subset of the boundary.
Suppose that $\rho$ is an elliptic  coefficient function.
Let $U_1,\ldots, U_n$ be an open covering of $\overline \Lambda$ and define $\Lambda_j = U_j \cap \Lambda$,
$N_j = U_j \cap (\partial \Lambda \setminus  E) \subset \partial \Lambda_j$, $E_j = \partial \Lambda_j
\setminus N_j$.\\
Let $q \ge 2$. 
Suppose that, for every $f_j \in W^{-1,q}_{E_j}(\Lambda_j)$, the solution $u_j \in W^{1,2}_{E_j}(\Lambda_j)$
of $-\nabla \cdot \rho|_{\Lambda_j} \nabla u_j = f_j$ in fact lies in $W^{1,q}_{E_j}(\Lambda_j)$. \\
Then the solution $u$ of $-\nabla \cdot \rho \nabla u= f \in W^{-1,q}_E(\Lambda)$ belongs to $W^{1,q}_E(\Lambda)$.
\end{theorem}

The point here is the following: For $\eta_j \in C^\infty_0(U_j)$ and $u \in W^{1,2}_E(\Lambda)$ one 
has $\eta_j u \in W^{1,2}_{E_j}(\Lambda_j)$. Moreover, $\eta_j u$ fulfills an analogous equation
$-\nabla \cdot \rho|_{\Lambda_j} \nabla (\eta_j u) + \eta_j u = f_j \in W^{-1,q}_{E_j}(\Lambda_j)$, see Lemma 
\ref{l-cutoff} below.\medskip

Relying on this localization principle we will be able to show $W^{1,q}$-regularity with $q>3$ of the solution of the elliptic problem for a wide range of geometric constellations by a reduction to a few model situations and the application of bi-Lipschitz transformations  (which are in most cases even $C^1$). 

%In order to distinguish the model situation from the general one, we will denote in the model situation the domain by $\Lambda\subset \R^3$,
% its Dirichlet boundary by $E$ and the coefficient function by $\rho$, while, in the general case, we will write $\Omega\subset \R^3$, $D$ and $\mu$ for the corresponding items.

In this context, we recall the following transformation theorem:  

\begin{proposition} \label{p-transform} 
Let $\Lambda, \Pi  \subseteq \R^3$ be open, bounded 
%with finitely many components 
and $E$ be a closed subset of $\partial\Lambda $.
Assume that $\phi:  \Lambda \to \Pi$ is bi-Lipschitz and define
$\widehat \phi (E) =: F$, $\widehat \phi$ being the bi-Lipschitz extension of $\phi$ to $\overline \Lambda$.
\renewcommand{\labelenumi}{\roman{enumi})}
\begin{enumerate}
\item \label{p-transform-1} 
$\phi$ induces a {\rm (}consistent in $q \in [1,\infty[)$ set of linear topological isomorphisms
\[%begin{equation} \label{e-litopis}
\Psi_q : W_{F}^{1,q}(\Pi) \to W^{1,q}_{E}(\Lambda).
\]%end{equation}
given by  $(\Psi_q f )(\x) = f(\phi(\x)) = (f \circ \phi)(\x)$.\\

\item
Let $\rho$ be a coefficient function on $\Lambda$.
For every  $f,g \in W^{1,2}(\Pi)$  the formula 
\begin{equation} \label{e-transfromul}
\int_\Lambda \rho (\x) \nabla \bigl ( f \circ \phi \bigr )(\x) \cdot 
 \nabla \bigl ( \overline g \circ \phi \bigr )(\x)\; d \x =
\int_{\Pi}  {\omega} (\y) \nabla f (\y) \cdot 
 \nabla \overline g (\y)\; d \y, 
\end{equation}
holds with
\begin{equation}\label{e-transformat1}
	  \omega(\y) = (\mathcal D\phi)(\phi^{-1}(\y)) \rho(\phi^{-1}(\y)) \bigl (\mathcal D\phi \bigr)^T(\phi^{-1}(\y))
		\frac{1}{\bigl|\det (\mathcal D\phi)(\phi^{-1}{\y})\bigr|},
	\end{equation}
	where $\mathcal D\phi$ denotes the Jacobian of $\phi$ and $\det(\mathcal D\phi)$ the
	corresponding determinant.
\item 
 The following formula holds:
	\begin{equation}\label{e-transformat}
	 \bigl ( \Psi_{q'}\bigr )^* \nabla \cdot \rho \nabla \Psi_q = \nabla \cdot \omega \nabla.
	\end{equation}
	Finally, if $-\nabla \cdot \rho \nabla : W^{1,q}_E(\Lambda)
	\to W^{-1,q}_E(\Lambda)$ is a topological isomorphism, then
	$-\nabla \cdot  \omega \nabla :W^{1,q}_{F}(\Pi) \to W^{-1,q}_{F}(\Pi)$ also is (and vice versa).
\end{enumerate}
\end{proposition}

\begin{proof}
i) The assertion for $E = \emptyset$ is proved in \cite[Section 1.1.7]{mazyasob}
 in  case $\Lambda$ is connected. This carries over to open sets when considering the connected components separately. 
This shows that $  \Psi_q$ maps $W^{1,q}_F(\Pi)$ -- as a subspace of  $W^{1,q}(\Pi)$ -- continuously into 
 $W^{1,q}(\Lambda)$. It remains to prove that $\Psi_qf\in W^{1,q}_E(\Lambda)$ if $f \in  W^{1,q}_F(\Pi)$.
It suffices to show this for $f \in C^\infty_F(\Pi)$, because $W^{1,q}_E(\Lambda)$ is a 
\emph{closed} subspace of $W^{1,q}(\Lambda)$. For such $f$, the function $\Psi_qf = f \circ \phi$
 is Lipschitzian, and its support has positive distance to $E$. 
By classical results (cf. \cite[Theorem 3.1]{ev/gar}), the  function $\Psi_qf$  has a Lipschitz continuous extension $g$ 
to $\R^3$. One can easily achieve that the support of $g$ also has positive distance to $E$. Using a mollifier argument, 
$g$ is the limit in $W^{1,q}(\R^3)$ for a sequence $(g_n)$ of functions from $C^\infty_0(\R^3)$ whose supports also have
 positive distance to $E$. Thus, $g_n|_\Lambda \in C^\infty_E(\Lambda)$ and $(g_n|_{\Lambda})$ evidently converges to $g|_\Lambda =\Psi_qf$ in $ W^{1,q}(\Lambda)$.  

ii) The formulas \eqref{e-transfromul} and \eqref{e-transformat1} have been derived 
in \cite[Proposition 16]{haller} using the the rules for the (weak) differentiation
of the compositions  $f \circ \phi$ and $ g \circ \phi$, respectively, cf. \cite[Section 1.1.7]
{mazyasob}.\\
iii) Formula \eqref{e-transformat} is directly implied by \eqref{e-transfromul}. The last
assertion follows from (i) by duality and \eqref{e-transformat}.
\end{proof}
\begin{definition} \label{d-gefolgt}
From now on we call $\omega$ the \emph{coefficient function obtained from $\rho$ by means
of the transformation} $\phi$, or, in short, the \emph{transformed coefficient function}.
\end{definition}

\subsection{The geometric setting} 
We next introduce the ``buried'' part ${D^\parallel}$ of the boundary and its closure $\overline{D^\parallel}$, of which we distinguish four subsets.  See Figure \ref{fig:Model} for a visualization.  We formulate the assumptions we impose on these sets and consider several (known) model constellations. 

%\begin{assu} \label{a-assugeneral}
%In the sequel we fix a bounded domain $\Omega \subset \R^3$ and a closed 
%part of its boundary, $D$, which has 3-dimensional Lebesgue measure $0$. We let %$N$ we denote the set
%$$N= \partial \Omega \setminus D.$$ 
%Moreover, we fix an \emph{elliptic} coefficient function $\mu$ on $\Omega$, cf. Def. \ref{d-coeff}.
%\end{assu}

\begin{definition} \label{d-typoparall}
$D^\parallel \subset D$ is the set of all points $\x\in D$, for which $\Omega \cup D$ is a
 neighbourhood of $\x$ in $\R^3$;  in other words: 
\begin{equation*}%\label{Dparallel}
D^\parallel  = \text{\rm int}(\Omega \cup D) \setminus
\Omega.
\end{equation*}
\end{definition}

Clearly, $D^\parallel \not=\emptyset$ if and only if $\Omega$ is a proper subset of $\widehat \Omega=\Int(\Omega\cup D)$. 

\begin{lemma} \label{l-open}$D^\parallel$ is open in $D$ and every point $\x \in D^\parallel$ has  positive distance to $N$. In particular, $D^\parallel\cap \overline N=\emptyset$. 
\end{lemma}

\begin{proof}
For $\x  \in D^\parallel$, there exists a  ball $B(\x,r)$ which is contained in $\Omega \cup D$.
This implies that $B(\x,r)\cap D\subset D^\parallel$; moreover it shows that $\dist(\x,N)\ge r$, since $N$ is
 disjoint to $\Omega \cup D$.
\end{proof}

It turns out that it makes sense to divide $D^\parallel$ into the following two subclasses:

\begin{definition} \label{d-rigidincl0}
$\x \in D^\parallel $ belongs to $ D^\parallel _c$, if $B(\x,r) \cap \Omega$ is connected for all $r$ from a nonempty
interval $]0,r_0[$. 
We let $D_d^\parallel := D^\parallel \setminus D^\parallel _c$.
\end{definition}

Before formulating the assumptions on the points in $D^\parallel $, it is our intention to point out already here 
the following topological implication for $D^\parallel_d$. 

\begin{lemma} \label{l-disjointopen}
Let $\Lambda_1,\ldots, \Lambda_m$ be mutually disjoint open sets in $\R^3$ and $\Lambda$ their union. Then
\begin{equation} \label{e-boundaryident}
\partial \Lambda_j \subseteq \partial \Lambda, 
\quad j \in \{1,\ldots,m \}.
\end{equation}
\end{lemma}
\begin{proof}
Suppose this is false. Then there is an index $j$ and a point $\z \in \partial \Lambda_j$ which does not belong to 
$\partial \Lambda$. 
%Since $\z$ can, obviously, not belong to  $\R^3 \setminus \overline \Lambda$ this would imply 
%$\z \in \Lambda$, since $\Lambda$ was open. 
As $\partial \Lambda_j\subseteq \overline{\Lambda}_j\subseteq\overline{\Lambda}$ and $\z\notin \partial \Lambda$, we have $\z\in \Lambda$. Hence, there  exists an open ball $B \subset \R^3$ around $\z$, 
which entirely lies in $\Lambda= \cup _l \Lambda_l$. 
Since the sets $\Lambda_l\cap B$ are open and disjoint and cover the connected set  $B$, there exists a $k$ such that $B\subseteq \Lambda_k$ and $B\cap\Lambda_l=\emptyset$ for $l\not= k$.   
But then $\z$ is not in $\partial \Lambda_j$ for any $j$, contradicting the assumption.  
\end{proof}
\begin{lemma} \label{l-01}
Suppose $U \subset \Omega \cup D$ for an open set $U \subset \R^d$ and that
 $\Omega \cap U= \cup_{j=1} ^m \Omega_j $, where the sets $\Omega_j$ are mutually disjoint and connected, in other 
words: the $\Omega_j$ are the -- finitely many -- connected components of $\Omega \cap U$.
Then 
\begin{equation} \label{e-connex1}
\partial {\Omega_j}  \subset D \cup \partial U, \quad j \in \{1,\ldots,m \}
\end{equation}
and
\begin{equation} \label{e-connex2}
\partial {\Omega_j} \cap U \subset D,\quad j \in \{1,\ldots,m \}.
\end{equation}
\end{lemma}
\begin{proof}
According to \eqref{e-boundaryident}, one may write (see \cite[Section 3.8]{dieu})
\begin{eqnarray*}
\partial \Omega_j  &\subseteq& \partial (\Omega \cap U) 
\subseteq  (\partial \Omega \cap U ) \cup  (\partial \Omega \cap \partial  U ) \cup  (\Omega \cap \partial  U )\\
&\subseteq& 
\bigl  (\partial \Omega \cap ( \Omega \cup D)  \bigr ) \cup \partial U = D \cup \partial U,
\end{eqnarray*}
because $\Omega$ is open, i.e. $\Omega \cap \partial \Omega = \emptyset$.\\
\eqref{e-connex2} is obtained from \eqref{e-connex1} by intersecting with $U$ and taking into account that 
$U$ is open, i.e. $\partial U \cap U = \emptyset$.\\
\end{proof}

\begin{corollary} \label{c-boundconn}
Adopt the assumptions of Lemma \ref{l-01}. 
For any two functions $f \in W^{1,q}_D(\Omega), \eta \in C^\infty_0(U)$, the function $\eta f|_{\Omega_j}$ belongs 
to $W^{1,q}_0(\Omega_j)$.
\end{corollary}

\begin{proof}
It is clear that it suffices to show this, by density, only for functions $f \in C^\infty_D(\Omega)$.
\eqref{e-connex1} shows that the support of $\eta f$ stays away from $\partial \Omega_j$
because the support of $\eta$ is by assumption away from $\partial U$, and that of $f$ is away from $D$.
\end{proof}
Later on we will discuss the regularity for the solution of the elliptic equation just by considering functions
$\eta f|_{\Omega_j}$ with $f \in W^{1,2}_D(\Omega), \eta \in C_0^\infty(U)$. 
The above considerations make clear that \emph{a priori} these
functions do belong to $W^{1,2}_0(\Omega_j)$, i.e. fulfill a pure Dirichlet condition.
Exactly this motivates the subsequent assumption on the points in $D^\parallel_d$.

\begin{assu} \label{a-00}
\renewcommand{\labelenumi}{\roman{enumi})}
\begin{enumerate}
\item 
For $\x \in D^\parallel_d$ there is an open connected neighborhood $U_{\x }  \subset \Omega \cup D$ of $ \x$,
not touching $\partial \Omega \setminus D$, 
 such that  $\Omega \cap U_{\x}$ is the disjoint union of finitely many connected components $\Omega_1, \ldots, 
\Omega_{m}$, each of  them being a Lipschitz domain with the property that $\x \in \cap_{j=1}^{m} \overline {\Omega}_j$.
\item
For $\mu$ as in Section \ref{a-assugeneral}, let $\mu_j = \mu|_{\Omega_j}$. Then 
\begin{equation} \label{e-localisodiri}
- \nabla \cdot \mu_j \nabla : W^{1,q}_0 (\Omega_j)  \to W^{-1,q} (\Omega_j), \quad j =1,\ldots , m, 
\end{equation}
is a topological isomorphism for some $q > 3$.
\end{enumerate}
\end{assu}
At the end we need the isomorphism property \eqref{e-localisodiri} for $\x \in D^\parallel_d$. 
We are aware that it would be desirable to give geometric conditions on $\Omega, D$ and on the coefficient
function $\mu$ which are are \emph{necessary and sufficient} for \eqref{e-localisodiri} being an isomorphism. 
Since such conditions seem to be unknown, we continue by referring  to \emph{examples} where the isomorphism property is known to hold.

\begin{proposition} \label{p-graphdomain}
Let $\Lambda \subset \R^3$ be a bounded Lipschitz graph domain {\rm (}cf. \cite[Definition 1.2.1.1]{grisvard85}; 
equivalently: $\Lambda$ is a strong Lipschitz domain in the sense of  \cite[Section 1.1.9]{mazyasob}; equivalently:
$\Lambda$ possesses the uniform cone property, cf. \cite[Theorem 1.2.2.2]{grisvard85}{\rm)}.
Suppose the coefficient function $\rho$ is elliptic, uniformly continuous and attains only values in \emph{real, symmetric}
matrices. Then there is a $p>3$ such that 
\begin{equation} \label{e-grraph}
-\nabla \cdot \rho \nabla :W^{1,q}_0(\Lambda)  \to W^{-1,q}(\Lambda)
\end{equation}
is a topological isomorphism for all $q \in [2,p[$, cf. \cite[Theorem 3.12]{ers}, compare also \cite[Theorem 0.5]{jer/ke}
 for the case of the Dirichlet problem for the Laplacian.
\end{proposition}

\begin{proposition} \label{p-convexhetero}
Let $\Lambda \subset \R^3$ be a bounded Lipschitz domain whose  closure  is a polyhedron.
Let $\Pi$ be any plane in $\R^3$ which intersects $\Lambda$, and assume that the angles between $\Pi$ and all (parts of) 
adjacent boundary planes are not larger than $\pi$. Suppose, moreover,  that the (elliptic) coefficient function $\rho$ 
is constant on  each of the components of $\Lambda \setminus  \Pi$.\\
Then there is a  $p >3$, such that \eqref{e-grraph} is a topological isomorphism for all $q \in {[2,p[}$, cf. \cite[Theorem 2.1]{e/k/r/s}.
\end{proposition}

\begin{corollary} \label{e-corspiegel}
%ES Let $\blacksquare \subset \R^2$ be any (open) rectangle and $\Lambda$ be the cuboid $\blacksquare \times {]-1,0[}$.
Let $\Box \subset \R^2$ be any (open) rectangle and $\Lambda$ be the cuboid $\Box \times {]-1,0[}$.
 Put $M:= \Box \times \{0\}$ and $E:=\partial \Lambda  \setminus M$. Let the elliptic coefficient function 
$\rho$ be constant on $\Lambda$.
Then there is a $p>3$, such that 
\begin{equation} \label{e-Neumanntop}
-\nabla \cdot \rho \nabla: W^{1,q}_E(\Lambda) \to W^{-1,q}_E(\Lambda)
\end{equation}
is a topological isomorphism for $q \in [2,p[$.
\end{corollary}
\begin{proof}
Reflection across the $(x,y)$-plane (compare \cite[Proposition 4.13]{mitRobert}) allows to transform the
problem \eqref{e-Neumanntop} to a pure Dirichlet problem which fits into Proposition \ref{p-convexhetero}.
\end{proof}

\begin{assu} \label{d-rigidincl}
For every point $\x \in D_c^\parallel$ there is an open neighborhood $ U_{\x} \subset \Omega \cup D$
for which the following conditions are satisfied:
\renewcommand{\labelenumi}{\roman{enumi})}
\begin{enumerate}
\item
There is a $C^1$-diffeomorphism $\phi_\x$ from $ U_{\x}$ onto the cube ${\bf C}$ (see Subsection \ref{localization}),
 satisfying
\begin{equation} \label{e-2} 
\phi_\x(\Omega \cap U_{\x})= {\bf  C } \setminus \Sigma, \quad
\phi_\x( U_{\x} \cap D^\parallel)= %\phi_\x( U_{\x} \cap D_c^\parallel)=
\Sigma, \quad \phi_\x(\x)=0,
\end{equation} 
where $\Sigma$ is one of the sets $\Sigma_1, \Sigma_2, \Sigma_3$, below:

\begin{equation} \label{e-halbflaeche}
\Sigma_1:=\{\y:y_1=0, -1 < y_2 \le 0, -1 < y_3 < 1\}
\end{equation}

\begin{eqnarray}
\label{e-halbflaeche00}
&&\Sigma_2 \; \text{is the \emph{closed} triangle with vertices }\\
&& (0,0,0),  (0,-1,-1),  (0,-1,1)\nonumber\\
&& \text{minus the (open) leg
between}\, (0,-1,-1) \text{and} \,(0,-1,1) .\nonumber
\end{eqnarray}

\begin{equation} \label{e-halbflaeche0}
\Sigma_3:=\bigl ({\bf  C } \cap \{\mathrm z: z_1=0 \}\bigr ) \setminus \mathop{\rm Int}(\Sigma_2). 
\end{equation}
\item
The limit
$\lim_{\mathrm z \to \x, \mathrm z \in \Omega} \mu(\mathrm z) =:\mu_{\x}$ exists.
\end{enumerate}
\end{assu}

The aim of this paper is to prove elliptic regularity not only around the points from $D^\parallel$, 
but also around the closure of $D^\parallel$. Therefore, we introduce the following

\begin{definition} \label{d-typo2}
Consider the closure $\overline {D^\parallel}$ of $D^\parallel$ in $\partial \Omega$. In the sequel, we will denote  
$$R= \overline {D^\parallel} \setminus D^\parallel.$$ 
Analogously to $D^\parallel$, the set $R$ is divided into the subsets $R_c$ and
 $R_d$, where  $R_c$ is the set of points $\x$ such that $\Omega \cap B$ is connected for any ball $B$ around $\x$ with 
sufficiently small radius, and $$R_d:=R \setminus R _c.$$
\end{definition}
Note that $R \subseteq D$, since $D$ is a \emph{closed} subset of  $\partial \Omega$ and $D^\parallel \subset D$.
But, in contrast to the points of $D^\parallel$, it may happen here that $D$ and $N$ \emph{do touch} in points of 
$R$.
Therefore one must be careful in formulating the analytical conditions on the points of $R$.

\begin{assu} \label{a-01}
\renewcommand{\labelenumi}{\roman{enumi})}
\begin{enumerate}
\item 
For $\x \in R_d$ there is an open, connected neighbourhood $U_{\x} $ of $\x$ such that
 the set $\Omega \cap U_\x$ splits up into at most finitely many connected components $\Omega_1, \ldots, \Omega_{m}$,
each of them being a Lipschitz domain.
Moreover, for every pair of different indices $i,j$ one has $\x \in \overline {\Omega_j} \cap \overline {\Omega_i }\subseteq D$.
\item
Let $N_j=N \cap \overline \Omega_j$ and  $D_j:= \partial \Omega_j \setminus N_j$. Assume
 that each $D_j$ is a $2$-set, see \eqref{e-twoset}.
If $\mu$ is an elliptic coefficient function on $\Omega$, and $\mu_j$ is the restriction to $\Omega_j$, then
\[%begin{equation} \label{e-localiso}
- \nabla \cdot \mu_j \nabla : W^{1,q}_{D_j} (\Omega_j)  \to W^{-1,q}_{D_j}  (\Omega_j), \qquad j=1,\ldots , m,
\]%end{equation}
is a topological isomorphism for some $q > 3$.
\end{enumerate}
\end{assu}

The following proposition shows model constellations where Assumption \ref{a-01} is fulfilled that will be of interest in the sequel.

\begin{proposition} \label{p-01} 
%ES Let $\blacktriangle {\color{magenta}\triangle ?}  \subset \R^2$ be an open triangle with vertices $P_1, 
Let $\triangle \subset \R^2$ be an open triangle with vertices $P_1,  P _2, P_3$.
For given real numbers $a, b$ define $\Lambda:= \triangle  
\times\left]a, b \right[$.\\
Let furthermore  $P $ denote the midpoint of the open segment
$\overline {P_1P_2}$, and let the (open) boundary part $\Upsilon_2$ be
\begin{equation} \label{e-Neummanrand}
\overline {P_1 P_2}\times {]a,b[} \;\;\;
\text{or} \;\;\;  \overline {P_1 P} \times{ ]a,b[ },
\end{equation}
see Figures $3$ and $4$, below, %{\color{red}  \ref{fig:2.17.1} and \ref{fig:2.17.2}},
and set $E = \partial \Lambda \setminus \Upsilon_2$.
Suppose  $\mathcal H$ to be a plane within $\R^3$ that intersects $\Lambda$,
but has a positive distance to the `ground plate' $\triangle\times \{a\}$ and the `cover plate' 
$\triangle\times \{b\}$. If the elliptic  coefficient function $\rho$ is 
constant on both components of $\Lambda \setminus \mathcal H$, then there is a $p>3$ such that
  \[%begin{equation} \label{e-iso05}
-\nabla \cdot \rho \nabla :W^{1,q}_{E}(\Lambda) \to W^{-1,q}_{E}(\Lambda)
  \]%end{equation}
  is a topological isomorphism for all $q \in [2,p[$.
\end{proposition}
\begin{proof} The results are proved in \cite[Theorem 1 and Theorem 2]{haller}.
\end{proof}

\begin{figure}[h]\label{fig:2.17.1}
    \centerline{\includegraphics[scale=0.5]{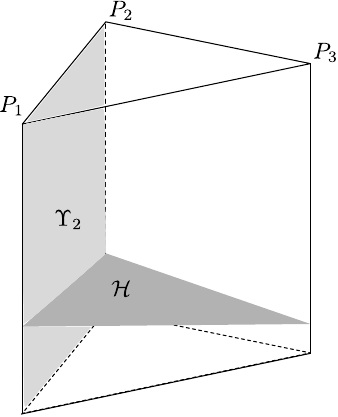}}
     \caption{The model set for the first case in \eqref{e-Neummanrand} }
  \end{figure}

\begin{figure}[h]\label{fig:2.17.2}
    \centerline{\includegraphics[scale=0.5]{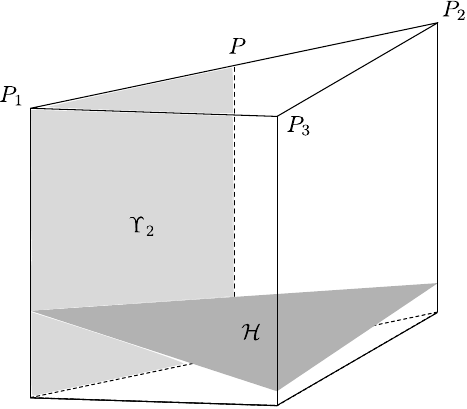}}
     \caption{The model set for the second case in \eqref{e-Neummanrand}}
  \end{figure}

\begin{rem} \label{r-movement} 
Recall that the situation described in Proposition \ref{p-01} may be carried over by bi-Lipschitz transformations
$\phi:\Lambda \to \Xi$ with $F =\phi (E)$.\\
Of particular interest are here mappings which are piecewise $C^1$, since then the subdomains of 
continuity for the coefficient function are mapped onto subdomains of $\Xi$ where the 
(transformed) coefficient function again is continuous.
\end{rem}

In a next step, we will specify the analytical assumptions on the local geometry of $\Omega$ around points in $R_c$.
\begin{assu} 
\label{d-Rout}
For every point $\x \in R_c $, there is an open, connected neighborhood $ U_{\x}$, which satisfies the 
following conditions:
\renewcommand{\labelenumi}{\roman{enumi})}
\begin{enumerate}
\item
There is a $C^1$-diffeomorphism $\phi_\x: U_{\x} \to {\bf C}$ such that 
\begin{equation} \label{e-019}
\phi_\x( U_{\x} \cap \Omega)= {\bf Q} \setminus \Sigma_1, \quad
\phi_\x( U_{\x} \cap D^\parallel)=
\Sigma_1\cap {\bf H}_3^- , \quad
\phi_\x(\x)=0,
\end{equation} 
with $\Sigma_1$  defined in \eqref{e-halbflaeche}.
\item
Denoting the transformed (under $\phi_\x$) coefficient function on 
$ {\bf Q } \setminus \Sigma_1 $ by $\underline \mu$, we require that 
both limits $\lim_{\mathrm z \in {\bf Q} _\pm ,
\mathrm z \to 0} \underline \mu(\mathrm z)=:\mu_\pm$ exist and are related by 
the involution $\iota$ from Section \ref{ss-prelim}: 
\begin{equation} \label{e-transiota}
\mu_-=\iota \mu_+ \iota.
\end{equation}
\item
Unless $ U_\x \cap N$ is  empty we demand
\[
\phi_\x( U_\x \cap N)= ({\bf C} \cap \{\mathrm z :z_3=0\}) \setminus \Sigma_1,
\]
(In this case the whole top surface - minus $\Sigma_1$ -- is  the (local) Neumann part of the
boundary.)
 \\
or
\[
\phi_\x(U_\x \cap N)= ({\bf  C} \cap \{\mathrm z :z_3=0\})\cap {\bf H}_2^+.
\]
(In this case half of the top surface is the (local) Neumann boundary part.)
\end{enumerate}
\begin{figure}[htbp]
\centering
\begin{tikzpicture}[
    scale=1.8,
    % Achsenrichtungen für Kavaliersperspektive:
    % x = Westsüdwest, skaliert mit Faktor 0,6
    % y = Osten (nach rechts)
    % z = Norden (nach oben)
    x={(-0.5543cm, -0.2296cm)}, 
    y={(1cm, 0cm)}, 
    z={(0cm, 1cm)},
    mystyle/.style={line width=1.2pt}
]
    % Eckpunkte des Quaders Q = [-1,1] x [-1,1] x [-1,0]
    % Untere Fläche (z = -1)
    \coordinate (A) at (-1, -1, -1); 
    \coordinate (B) at ( 1, -1, -1); 
    \coordinate (C) at ( 1,  1, -1); 
    \coordinate (D) at (-1,  1, -1); 
    
    % Obere Fläche (z = 1)
    \coordinate (E) at (-1, -1,  0); 
    \coordinate (F) at ( 1, -1,  0); 
    \coordinate (G) at ( 1,  1,  0); 
    \coordinate (H) at (-1,  1,  0); 

    % Eckpunkte der ersten schraffierten Fläche [-1,1] x [-1,1] x {0} (auf der Oberseite)
%    \coordinate (S1) at (-1,  -1,  0);
%    \coordinate (S2) at ( 1,  -1,  0);
%    \coordinate (S3) at ( 1,  1,  0);
%    \coordinate (S4) at (-1,  1,  0);

    % Eckpunkte der zweiten gepunkteten Fläche Sigma1= {0} x ]-1,0] x ]-1,0] 
    \coordinate (T1) at ( 0,  -1, -1);
    \coordinate (T2) at ( 0,  0, -1);
    \coordinate (T3) at ( 0,  0,  1);
    \coordinate (T4) at ( 0,  -1,  1);
    \coordinate (T5) at ( 0,  0,  0);
    \coordinate (T6) at ( 0,  -1,  0);

    % 1. Verdeckte Kanten des Quaders (gestrichelt)
    \draw[dashed, gray, thick] (A) -- (B);
    \draw[dashed, gray, thick] (D) -- (A);
    \draw[dashed, gray, thick] (A) -- (E);

    % 2. Beide schraffierten Flächen ZUERST zeichnen
%    \draw[pattern=north east lines, pattern color=blue!50!black, mystyle, draw=blue!60!black] 
%        (S1) -- (S2) -- (S3) -- (S4) -- cycle;
    \draw[pattern=dots, pattern color=black, dashed, draw=black] 
        (T1) -- (T2) -- (T3) -- (T4) -- cycle;
    \draw[thick] (T2) -- (T5) -- (T6);
    
    % 3. Sichtbare Kanten des Quaders (durchgezogen)
    % Untere sichtbare Kanten
    \draw[thick] (B) -- (C) -- (D);
    % Obere Deckfläche (z = 0) - Umrandung
    \draw[thick] (E) -- (F) -- (G) -- (H) -- cycle;
    % Vertikale Kanten
    \draw[thick] (B) -- (F);
    \draw[thick] (C) -- (G);
    \draw[thick] (D) -- (H);

    % 4. Koordinatenachsen ab Ursprung O(0,0,0)
    \draw[->, thick, red!70!black] (0,0,0) -- (1.9,-0.2,0) node[anchor=north east] {$x$ };
    \draw[->, thick, red!70!black] (0,0,0) -- (0,1.7,0) node[anchor=south west] {$y$};
    \draw[->, thick, red!70!black] (0,0,0) -- (0,0,1.2) node[anchor=south east] {$z$ };

    % Ursprung markieren
    \fill[red!70!black] (0,0,0) circle (1pt);
    \node[red!70!black] at (0,0,0) [anchor=north east] {$O$};

    % Beschriftung der Eckpunkte
%    \node[font=\footnotesize, below left] at (A) {$A$};
%    \node[font=\footnotesize, below right] at (B) {$B$};
%    \node[font=\footnotesize, above right] at (C) {$C$};
%    \node[font=\footnotesize, above left] at (D) {$D$};
%    \node[font=\footnotesize, left] at (E) {$E$};
%    \node[font=\footnotesize, right] at (F) {$F$};
%    \node[font=\footnotesize, above right] at (G) {$G$};
%    \node[font=\footnotesize, above left] at (H) {$H$};

    % Titel der Zeichnung
%    \node[above] at (0, 0, 1.5) {\textbf{};
  \draw[thin,black] (0, -0.5, 0.4) -- (0.8, 0.5, 0.7) 
        node[right, align=left, font=\small] {$\Sigma_1 =\{0\} \times {]-1,0]} \times {]-1,1[}$};

\end{tikzpicture}
\caption{The set $\mathbf Q\setminus \Sigma_1$  }
\label{fig:Eq2.18}
\end{figure}
\end{assu}

\begin{rem} \label{r-fortlassen}
Suppose $\mu_- = \iota \mu_+\iota$ and define the coefficient function 
$\check \mu$ on ${\bf Q  }$ by 
\[
\check \mu (\mathrm z) = \begin{cases} 
\mu_+,& \text{if} \; \mathrm z \in {\bf  Q}_+ \\
\mu_- ,& \text{if} \; \mathrm z \in {\bf Q}_- \\
{\rm Id}_{\R^3}  ,& \text{ on} \; {\bf Q } \cap \{\mathrm z: z_1 =0 \}\
\end{cases}.
\]
Then $\check \mu$ satisfies the crucial relation 
\begin{equation} \label{e-crucial}
\check \mu(\iota(\mathrm z))=\iota \check \mu (\mathrm z)\iota, \quad \mathrm z 
\in {\bf  Q}_+ \cup {\bf Q}_-.
\end{equation}
The essential point is that the coefficient function, resulting from $\check \mu$ 
by the transformation $\iota$, remains $\check \mu$, cf. Proposition \ref{p-transform} above.\\
Of course, the combination of the mapping properties of $\phi_\x$ and condition \eqref{e-transiota} for the
transformed matrix is indeed a restriction on the original constellation, see  (i) 
in the concluding remarks.
\end{rem}

\begin{proposition} \label{e-interpolatt}
Let $\Lambda \subset \R^3$ a bounded domain and $E\subset \partial \Lambda $ a $2$-set
of positive boundary measure. Suppose that there exists a linear, 
bounded extension operator $\mathfrak E:W^{1,q}_E(\Lambda) \to W^{1,q}_E(\R^3) $.
Let $\rho$ be an elliptic coefficient function. Then
\begin{equation} \label{e-sisomoro}
-\nabla \cdot \rho \nabla: W^{1,q}_E(\Lambda) \to 
 W^{-1,q}_E(\Lambda)
\end{equation}
is a topological isomorphism for $q=2$. The set of 
$q$'s, for which \eqref{e-sisomoro} is a topological
isomorphism, forms an \emph{open} interval.
\end{proposition}
\begin{proof}
The case $q=2$ is implied by Lax-Milgram because the form 
\[
W^{1,2}_E(\Lambda) \ni u \mapsto \int_\Lambda \rho
\nabla u \cdot \nabla u 
\]
is coercive due to the positive boundary measure of $E$. (In fact, a Poincar\'e inequality holds, see 
\cite[Ch.~7, Lemma 7.3]{EgertHardy}.)
The second assertion follows from the interpolation properties
of the scales $\{W^{1,q}_E(\Lambda)\}_{q \in {]1,\infty[}}$,
$\{W^{-1,q}_E(\Lambda)\}_{q \in {]1,\infty[}}$, respectively, 
see \cite[Theorem  5.6]{jons}).
\end{proof}

\section[Main Result and Proof] {Main Result and Proof }\label{s-proof}

We are now in the position to formulate the main result of this paper:
\begin{theorem} \label{t-main}
Let the Assumptions {\rm \ref{a-00}, \ref{d-rigidincl}, \ref{a-01}, \ref{d-Rout}} be satisfied, 
and denote by  $u \in W^{1,2}_D(\Omega)$ the solution of 
\[%begin{equation} \label{e-solutt}
-\nabla \cdot \mu \nabla u =f \in W^{-1,2}_D(\Omega)
\]%end{equation}
given by the Lax-Milgram theorem.
Then, for every $\x \in \overline {D^\parallel}$, there is an open neighborhood $ W_\x $ of $ \x$ in $\R^3$ and
a $p=p(\x)>3$ such that for every $q \in [2,p[$ the following holds: For every 
$\eta \in C^\infty_0( W_\x)$, the function $\eta u$ belongs to $W^{1,q}_0(\Omega)$, if 
$\x \in D^\parallel $, and belongs to $W^{1,q}_D(\Omega)$, if $\x \in  \overline{D^\parallel}\setminus D^\parallel $, provided that 
$f \in W^{-1,q}_D(\Omega)$.
\end{theorem}

%\section[The Proof] {The Proof }\label{s-Proof} 
%{\bf einiger einleitender Text waere wohl gut}
The proof is structured as follows: 
We start with several results on localization and freezing of coefficients. We  then consider a variety of model constellations adapted to the decomposition of $\overline {D^\parallel}$ into $D^\parallel_d$, $D^\parallel_c$, $R_d$ and $R_c$. From these we deduce the main theorem with the help of the transformation result, Proposition  \ref{p-transform}.

\subsection{Localization principles}
The next lemma provides the possibility of localizing the elliptic problem:

\begin{lemma} \label{l-cutoff}
 Let $\Lambda \subset \R^3$ be a bounded Lipschitz domain, $E$  a closed subset of the boundary, and 
$ V \subset \R^3$  bounded and open.  Let $\rho$ be an elliptic coefficient function on $\Lambda$.
Putting $M:=\partial \Lambda \setminus E$, define 
$$\Pi:=\Lambda \cap V , \quad 
N:=M \cap V ,\quad F:=\partial \Pi \setminus N.$$
Fix an arbitrary function $\eta \in C^\infty_0(V)$.% with support in $ V $.
Then, for every $q \in [1,\infty[$, the following hold:
\renewcommand{\labelenumi}{\roman{enumi})}
\begin{enumerate}
\item \label{l-cutoff-1}
If $v \in W^{1,q}_E(\Lambda)$, then $\eta v|_{\Pi} \in W^{1,q}_{F}(\Pi)$.
In particular, if $ V  \cap M =\emptyset$, then $\eta v \in W^{1,q}_0(\Pi)$.
\item  \label{l-cutoff-2}
For any function on $\Pi$ denote by  $\widetilde w$ the extension of $w$ to $\Lambda$ by $0$. 
Then the mapping
\[
W^{1,q}_{F}(\Pi) \ni v \mapsto \widetilde {\eta v}
\]
has its image in $W^{1,q}_E(\Lambda)$ and is continuous.
\item  \label{l-cutoff-3}
Suppose that, for $q \in {]1,3[}$, there is the usual embedding $W^{1,q}_{F}(\Pi)
\hookrightarrow L^{\frac {3q}{3-q}}(\Pi)$. 
Let $u \in W^{1,2}_E(\Lambda)$  be the solution of 
\[
-\nabla \cdot \rho \nabla u =f \in W^{-1,q}_E(\Lambda) \hookrightarrow  W^{-1,2}_E(\Lambda) , \; q \in [2,6].
\]
Then $v:=(\eta u)|_{\Pi} $ satisfies an equation
\[
-\nabla \cdot \kappa  \nabla v =g \in W^{-1,q}_{F}(\Pi)
\]
with $\kappa = \rho|_{\Pi}$. Moreover, for $q \in [2,6]$, the mapping 
$W^{-1,q}_E(\Lambda) \ni f \mapsto g \in W^{-1,q}_{F}(\Pi)$ is linear and continuous.
\end{enumerate}
\end{lemma}
\begin{proof}
See \cite[Lemma 5.8 and 5.9]{disser}.
\end{proof}

\begin{lemma} \label{l-localize}
Let $\Lambda \subset \R^3$ be a bounded domain and $E$ be a closed part
of its boundary. Suppose the validity of the embeddings $W^{1,q}_E(\Lambda) \hookrightarrow L^{\frac {3q}{3-q}}(\Lambda)$
for $ q \in [\frac {6}{5},2]$.\\
Let $\Lambda$ be the disjoint union $\Lambda=\Lambda_1 \cup \Lambda_2 \cup (\Lambda \cap  M)$, where $\Lambda_{1,2}\subset \Lambda$
are open, and $ M \subset \R^3$ is closed and Lebesgue negligible. 
Assume that  a fixed  $\x \in  M $ is an accumulation point 
for both  $\Lambda_1$ and $\Lambda_2$, and that $\lim_{{\y} \in \Lambda_j, {\y}\to {\x}}\rho({\y})=:\rho_j$, $j = 1,2$,
exist. Define the coefficient function $\hat \rho$ by 
\[
\hat \rho({\y})=\begin{cases}
\rho_j ,&  {\y} \in \Lambda_j, \;j =1,2 \\
 {\rm Id}_{\R^3},&{\y} \in  \Lambda \cap  M,
\end{cases}
\]
and suppose that $\nabla \cdot \hat \rho \nabla:W^{1,q}_E(\Lambda) \to
W^{-1,q}_E(\Lambda)$ is a topological isomorphism for a $q \in [2,6]$.
Then the following hold: 
\renewcommand{\labelenumi}{\roman{enumi})}
\begin{enumerate}
\item
If $ V$ is any sufficiently small neighborhood of $\x$ in $\R^3$, and $\rho$ is changed to
\[
 \rho_ { V}({\y})=\begin{cases}
\hat\rho(\y) &\y\in \Lambda \setminus  V\\
\rho(\y)  &  {\y} \in \Lambda \cap V.
\end{cases}
\]
Then 
\begin{equation} \label{e-iiso}
\nabla \cdot  \rho_ { V} \nabla:W^{1,q}_E(\Lambda) \to W^{-1,q}_E(\Lambda)
\end{equation}
 remains a topological isomorphism for the same $q$. 
\item 
For every sufficiently small neighbourhood $ V$ of ${\x}$ and any function $\eta \in C^\infty_0(V)$,
 the function $\eta u$ belongs to $W^{1,q}_E(\Lambda)$, provided that
$u \in W^{1,2}_E(\Lambda)$ satisfies $-\nabla \cdot \rho \nabla u =f \in W^{-1,q}_E(\Lambda)$.
\end{enumerate}
\end{lemma} 

\begin{proof}
i) Given $\epsilon >0$,  the assumption on $\rho$ implies that 
 $\| \rho _{ V}- \hat \rho \|_{L^\infty(\Lambda)}< \epsilon$
provided the neighbourhood $ V$ is sufficiently small. Taking into account the estimate
\begin{equation} \label{e-opnormgegenLinftiy}
\|\nabla \cdot (\hat \rho - \rho_ { V}) \nabla \|_{\mathcal L(W^{1,q}_E(\Lambda);W^{-1,q}_E(\Lambda))} \le 
\| \rho_ { V} - \hat \rho \|_{L^\infty(\Lambda)}
\end{equation}
and writing
\[
-\nabla \cdot \rho_ { V} \nabla = -\nabla \cdot \hat \rho \nabla +\nabla \cdot (\hat \rho - \rho_ { V}) \nabla ,
\]
this shows that $-\nabla \cdot  \rho_ {V}  \nabla $ is a perturbation of $- \nabla \cdot \hat \rho \nabla$,
and the difference of both can be made arbitrarily small in norm by taking $V$ small.
Hence, the first assertion follows by \eqref{e-opnormgegenLinftiy} and classical perturbation theory.\\
ii) Let $ V$ be a neighborhood of $\x$, such that \eqref{e-iiso} remains an isomorphism.
Further, let $ W$ be an open neighborhood of $\overline { \Lambda} \cup \overline { V}$ 
and $\eta$  be a Lipschitz continuous function  on $W$ with support in $ V \subset  W$.
Assume now that $u \in W^{1,2}_E(\Lambda)$ satisfies $-\nabla \cdot \rho \nabla u =f \in W^{-1,q}_E(\Lambda)$.
Applying Lemma \ref{l-cutoff}(iii)  one sees that the function $v:=\eta|_\Lambda u \in W^{1,2}_E(\Lambda)$ satisfies an equation
$-\nabla \cdot \rho \nabla v =g \in W^{-1,q}_E(\Lambda)$.
Clearly, $v$ vanishes identically  on $\Lambda \setminus  V$. 
 Consequently, $v$ also satisfies the elliptic equation
$-\nabla \cdot \rho_V \nabla v =g $. Then the isomorphism property \eqref{e-iiso} implies
the second assertion.
\end{proof}

\begin{lemma} \label{l-001}
Let $\Lambda_1, \dots,\Lambda_m$ be mutually disjoint bounded open sets in $\R^3$ and  $\Lambda := \cup_{j=1}^m \Lambda_j$. 
Moreover, let $E \subseteq  \partial  \Lambda$ be closed and  $E_j:= \partial \Lambda_j \cap E$.  
Then 
\renewcommand{\labelenumi}{\roman{enumi})}
\begin{enumerate}
\item
The different domains $\Lambda_j$ touch each other only across the (global) Dirichlet boundary part $E$.
\item
For  $u \in W^{1,2}_{E}(\Lambda )$ the function  $v:= u|_{\Lambda_j}$ belongs to $W^{1,2}_{ E_j }(\Lambda_j )$.
\item  $E$ is the union of the sets $E_j$, $j=1,\ldots,m$. 
\item
Suppose that $\overline {\Lambda_j} \cap \overline {\Lambda_i} \subseteq E$ for every pair of indices $i\not=j$. 
For fixed $j \in \{1,\ldots,m\}$ and $\psi \in C^\infty_{E_j}(\Lambda_j)$, let $\Psi \in C^\infty_E(\R^3)$ be any function whose restriction to $\Lambda_j$ equals $\psi$. Then $\mathop{\rm supp} \Psi \cap \overline \Lambda_j$ and 
$\mathop{\rm supp} \Psi \cap \bigl (\cup_{k \neq j} \overline {\Lambda_k}\bigr )$ have a positive distance to each other.
\end{enumerate}
\end{lemma}

\begin{proof}
i) and ii) are clear.
iii)  Since  $\partial \Lambda \subseteq \cup_j \partial \Lambda_j$, \eqref{e-boundaryident}
shows that $\partial \Lambda = \cup_j \partial \Lambda_j$ and hence  the assertion.
 iv) Suppose that the claim is false. Since both sets, $\mathop{\rm supp }\Psi\cap \overline \Lambda_j$
 and $\mathop{\rm supp }\Psi \cap \bigl (\cup_{k \neq j} \overline {\Lambda_k}\bigr )$ are compact, they must then possess a common point, say $\y$. According to the assumption that  $\overline {\Lambda_j} \cap \overline {\Lambda_i}
\subseteq E$, $\y$ then must belong to $E$.
But $\mathop{\rm supp }\Psi$ has an empty intersection with $E$ and, hence, also with $E_j$ -- which is a contradiction.
\end{proof}

\begin{lemma} \label{l-0x1}
Adopt the notation and assumptions of Lemma $\ref{l-001}$. Then:
\renewcommand{\labelenumi}{\roman{enumi})}
\begin{enumerate}
\item
Assume $\psi \in C^\infty_{E_j}(\Lambda_j)$ and let $\tilde \psi$ be the extension of $\psi$ by zero to a function
 on $\Lambda$. Then $\tilde \psi $ belongs to $C^\infty_{E}(\Lambda)$.
\item
\begin{equation} \label{e-isometr}
\|\tilde \psi \|_{W^{1,p}(\Lambda)} = 
\|\psi \|_{W^{1,p}(\Lambda_j )} \quad \text{for } \; p \in {]1,\infty[}.
\end{equation}
Hence, the mapping 
\[
C^\infty_{E_j}(\Lambda_j) \ni \psi \mapsto \tilde \psi \in C_{E }^{\infty}(\Lambda)
\]
extends by density to an isometric mapping
\[
W_{E_j}^{1,p}(\Lambda_j ) \ni \psi \mapsto \tilde \psi \in W_{E}^{1,p}(\Lambda).
\]
\end{enumerate}
\end{lemma}
\begin{proof}
i)
Let $\Psi \in C^\infty_{ E_j}(\R^3)$ be any function whose restriction to $\Lambda_j$ equals $\psi$.
Then, according to Lemma \ref{l-001}, one may take a function $\eta \in C^\infty_0(\R^3)$ which equals $1$ on 
$\mathop{\rm supp }\Psi\cap \overline \Lambda_j$ and $0$ on a neighbourhood of 
$\mathop{\rm supp }\Psi \cap \bigl (\cup_{k \neq j} \overline {\Lambda_k}\bigr )$.
Evidently, $\eta \Psi$ is then zero on a neighbourhood of $E$ and its restriction to $\Lambda$ is nothing else but
the zero-extension of $\psi$ to $\Lambda$.
ii) \eqref{e-isometr}  is clear and the rest follows.
\end{proof}

\begin{corollary} \label{e-restrictrhs}
Let us, in the terminology of the preceding lemma, associate to every $f \in W^{-1,q}_{E}(\Lambda)$
elements $ f_j \in  W^{-1,q}_{ E_j }(\Lambda_j )$ by defining 
\begin{equation} \label{e-deffj}
\langle  f_j, \psi \rangle := \langle f, \tilde \psi \rangle \quad \text{for} \quad \psi \in W^{1,q'}_{E_j}(\Lambda_j).
\end{equation}
Then the mapping 
\begin{equation} \label{e-differentcomp}
W^{-1,q}_{E}(\Lambda) \ni f \mapsto f_j \in  W^{-1,q}_{ E_j }(\Lambda_j ), \quad q \in [2,\infty[
\end{equation}
is well-defined and continuous. 
Moreover, for $f \in L^2(\Lambda ) \hookrightarrow W^{-1,2}_{E}(\Lambda )$, $f_j$ is just the restriction of $f$ to $\Lambda_j$.\\
 \end{corollary}
 
\begin{proof}
The mapping $f \mapsto f_j$ is the adjoint one to $\psi \mapsto \tilde \psi$. The last assertion is clear.
\end{proof}
\begin{lemma} \label{l-031}
Adopt the notation and assumptions of Lemma $\ref{l-001}$. Moreover, assume that each set $\Lambda_j$ is a 
Lipschitz domain and $E_j$  is of positive two-dimensional Hausdorff measure.
 If $u \in W^{1,2}_{E }(\Lambda)$ satisfies the equation
\[%begin{equation} 
- \nabla \cdot \rho \nabla u = f \in W^{-1,q}_{E }(\Lambda), \quad q \ge 2,
\]%end{equation}
then $v:= u|_{\Lambda_j}$ satisfies 
\begin{equation} \label{e-ggli}
- \nabla \cdot \rho|_{\Lambda_j} \nabla v =   f_j \in W^{-1,q}_{ E_j }(\Lambda_j).
\end{equation}
\end{lemma}
\begin{proof}
First of all, the assumptions on $\Lambda_j$ and $E_j$ admit a Poincar\'e inequality, see \cite[Ch. 7 Lemma 7.3]{EgertHardy}.
Hence, 
\[
- \nabla \cdot \rho|_{\Lambda_j} \nabla : W^{1,2}_{E_j}(\Lambda_j) \to W^{-1,2}_{E_j}(\Lambda_j)
\]
is a topological isomorphism.
By testing with $C^\infty_0(\Lambda_j)$ functions, one sees that
the assertion holds for $f \in L^2(\Lambda )$. This extends by the density of $L^2(\Lambda)$ in $W^{-1,2}_E(\Lambda)$
and the continuity of \eqref{e-differentcomp}, this considered for $q=2$, to all $f \in W^{-1,2}_E(\Lambda)$. 
This gives 
\[%begin{equation} \label{e-ggli}
- \nabla \cdot \rho|_{\Lambda_j} \nabla v =   f_j \in W^{-1,2}_{ E_j }(\Lambda_j).
\]
But from Cor. \eqref{e-restrictrhs} we know that that the $f_j$'s are better; they belong to  
$W^{-1,q}_{ E_j }(\Lambda_j)$. So indeed \eqref{e-ggli} holds, due to the injectivity of the mapping 
$W^{-1,q}_{ E_j }(\Lambda_j) \hookrightarrow W^{-1,2}_{ E_j }(\Lambda_j)$.
\end{proof}

\subsection{The relevant model constellations }
%\begin{color}{red}
%{\bf dieser Subtitel weckte das besondere Missvergnuegen des Referees. Wir muessen uns was anderes ausdenken...}
%\end{color}
In this section we are going to establish several results needed later on for the 
proof of Theorem \ref{t-main}. Since we have reason \emph{not} to work with the model constellations
established in Proposition \ref{p-01}, our first aim is to deduce from these regularity results
for slightly modified model problems which are  adequate  for later purposes.\\
In order to make the reading easier, we produced several graphics which show 
the geometry under consideration including the corresponding boundary parts. 
%Note that they only show the left halves of the model constellations  \eqref{e-2} (with $\Sigma = \Sigma_1$) and
%\eqref{e-019}, in order to make the `buried' part of the Dirichlet boundary visible. 

%They are to be read as follows: 
%Coordinate axes in $\R^3$ are chosen such that the $x$-axis points to the right, the $y$-axis backwards and the $z$-axis upwards. White surfaces always carry a Dirichlet boundary condition. The black surfaces stand for 
%points from $D^\parallel$, and, consequently, also represents a Dirichlet surface. Also the crosshatched                                 part is Dirichlet --
%resulting from antisymmetric reflection of the problem (see \eqref{e-gleich1} below).
%The grey part denotes the Neumann part; that is also true for the dotted one.
%The latter results from symmetric reflection of the original (model) problem, see \eqref{e-gleich2} below.

\begin{lemma} \label{c-modifymodel}
\renewcommand{\labelenumi}{\roman{enumi})}
\begin{enumerate}
\item
Define $M_- :=({]-1,0[} \times {]0,1[} \times \{0\}) \cup (\{0\} \times {]0,1[} \,\times {]-1,0]})$, 
$M_+=\iota M_-$,
and $E_\pm := \partial {\bf Q}_\pm \setminus M_\pm$.
Let $\rho_\pm$ be a \emph{constant} coefficient function on ${\bf Q}_\pm$.\\
Then there is a $p >3$, such that the mapping
\begin{equation} \label{e-iso5}
-\nabla \cdot \rho_\pm \nabla :W^{1,q}_{E_\pm}({\bf Q}_\pm) \to W^{-1,q}_{E_\pm}({\bf Q}_\pm)
\end{equation}
is a topological isomorphism for all $q \in [2,p[$.

%\begin{center}
\begin{figure}[htbp]
\centering
\begin{tikzpicture}[
    scale=2.0,
    % Achsenrichtungen für Kavaliersperspektive:
    % x = Westsüdwest, skaliert mit Faktor 0,6
    % y = Osten (nach rechts)
    % z = Norden (nach oben)
    x={(-0.5543cm, -0.2296cm)}, 
    y={(1cm, 0cm)}, 
    z={(0cm, 1cm)},
    mystyle/.style={line width=1.2pt}
]
    % Eckpunkte des Quaders Qminus = [-1,1] x [-1,1] x [-1,0]
    % Untere Fläche (z = -1)
    \coordinate (A) at (-1, -1, -1); 
    \coordinate (B) at ( 0, -1, -1); 
    \coordinate (C) at ( 0,  1, -1); 
    \coordinate (D) at (-1,  1, -1); 
    
    % Obere Fläche (z = 0)
    \coordinate (E) at (-1, -1,  0); 
    \coordinate (F) at ( 0, -1,  0); 
    \coordinate (G) at ( 0,  1,  0); 
    \coordinate (H) at (-1,  1,  0); 

    % Eckpunkte der ersten schraffierten Fläche [-1,0] x [0,1] x {0} (auf der Oberseite)
    \coordinate (S1) at (-1,  0,  0);
    \coordinate (S2) at ( 0,  0,  0);
    \coordinate (S3) at ( 0,  1,  0);
    \coordinate (S4) at (-1,  1,  0);

    % Eckpunkte der zweiten schraffierten Fläche {0} x [0,1] x [-1,0] (im Quaderinneren)
    \coordinate (T1) at ( 0,  0, -1);
    \coordinate (T2) at ( 0,  1, -1);
    \coordinate (T3) at ( 0,  1,  0);
    \coordinate (T4) at ( 0,  0,  0);

    % 1. Verdeckte Kanten des Quaders (gestrichelt)
    \draw[dashed, gray, thick] (A) -- (B);
    \draw[dashed, gray, thick] (D) -- (A);
    \draw[dashed, gray, thick] (A) -- (E);

    % 2. Beide schraffierten Flächen ZUERST zeichnen
    \draw[pattern=north east lines, pattern color=blue!50!black, mystyle, draw=blue!60!black] 
        (S1) -- (S2) -- (S3) -- (S4) -- cycle;
    \draw[pattern=north east lines, pattern color=blue!50!black, mystyle, draw=blue!60!black] 
        (T1) -- (T2) -- (T3) -- (T4) -- cycle;

    % 3. Sichtbare Kanten des Quaders (durchgezogen)
    % Untere sichtbare Kanten
    \draw[thick] (B) -- (C) -- (D);
    % Obere Deckfläche (z = 0) - Umrandung
    \draw[thick] (E) -- (F) -- (G) -- (H) -- cycle;
    % Vertikale Kanten
    \draw[thick] (B) -- (F);
    \draw[thick] (C) -- (G);
    \draw[thick] (D) -- (H);

    % 4. Koordinatenachsen ab Ursprung O(0,0,0)
    \draw[->, thick, red!70!black] (0,0,0) -- (1.2,0,0) node[anchor=north east] {$x$ };
    \draw[->, thick, red!70!black] (0,0,0) -- (0,1.8,0) node[anchor=south west] {$y$};
    \draw[->, thick, red!70!black] (0,0,0) -- (0,0,0.9) node[anchor=south east] {$z$ };

    % Ursprung markieren
    \fill[red!70!black] (0,0,0) circle (1pt);
    \node[red!70!black] at (0,0,0) [anchor= north east] {$O$};

    % Beschriftung der Eckpunkte
%    \node[font=\footnotesize, below left] at (A) {$A$};
%    \node[font=\footnotesize, below right] at (B) {$B$};
%    \node[font=\footnotesize, above right] at (C) {$C$};
%    \node[font=\footnotesize, above left] at (D) {$D$};
%    \node[font=\footnotesize, left] at (E) {$E$};
%    \node[font=\footnotesize, right] at (F) {$F$};
%    \node[font=\footnotesize, above right] at (G) {$G$};
%    \node[font=\footnotesize, above left] at (H) {$H$};

    % Titel der Zeichnung
%    \node[above] at (0, 0, 1.5) {\textbf{};

\end{tikzpicture}
\caption{The set $\mathbf Q_-$ in (i) with the set $M_-=({]-1,0[}\times {]0,1[}\times \{0\})\cup (\{0\}\times {]0,1[}\times {]-1,0]})$, where Neumann conditions are imposed (hatched).  The rest has Dirichlet boundary conditions. }
\label{fig:quader5.7.i}
\end{figure}
%\end{center}

%\begin{figure}[h]
%    \centerline{\includegraphics[scale=0.9]{BuriedBild1.pdf}}
%     \caption{${\bf Q}_-$ for \ref{c-modifymodel}.i). The grey and the dotted part form $M_-$ with Neumann
%b.c.,  the rest has Dirichlet b.c..}
%  \end{figure}

\item
Define $M= \{0\} \times ]0,1[ \times ]-1,1[$, and $E_- =\partial {\bf C}_- \setminus M$. 
Let $\rho$ be a coefficient function on ${\bf C}_-$ which is constant on the two subsets 
${\bf Q}_- $, and ${\bf C}_- \setminus {\bf Q}_-$, respectively. Then there is
a number $p>3$, such that 
\[
-\nabla \cdot \rho \nabla :W^{1,q}_{E_\pm }({\bf C}_\pm) \to W^{-1,q}_{E_\pm}({\bf C}_\pm)
\]
is a topological isomorphism for all $q \in [2,p[$.

%\begin{figure}[h]
%    \centerline{\includegraphics[scale=0.9]{BuriedBild5.pdf}}
%     \caption{${\bf Q}_-$ for \ref{c-modifymodel}.ii). The dotted part is  $\{0\} \times {]0,1[ }\times {]-1,0[}$ with Neumann b.c., the rest has Dirichlet b.c..}
%  \end{figure}
%\end{enumerate}
\begin{figure}[htbp]
\centering
\begin{tikzpicture}[
    scale=2.0,
    % Achsenrichtungen für Kavaliersperspektive:
    % x = Westsüdwest, skaliert mit Faktor 0,6
    % y = Osten (nach rechts)
    % z = Norden (nach oben)
    x={(-0.5543cm, -0.2296cm)}, 
    y={(1cm, 0cm)}, 
    z={(0cm, 1cm)},
    mystyle/.style={line width=1.2pt}
]
    % Eckpunkte des Quaders Cminus = [-1,0] x [-1,1] x [-1,1]
    % Untere Fläche (z = -1)
    \coordinate (A) at (-1, -1, -1); 
    \coordinate (B) at ( 0, -1, -1); 
    \coordinate (C) at ( 0,  1, -1); 
    \coordinate (D) at (-1,  1, -1); 
    
    % Obere Fläche (z = 1)
    \coordinate (E) at (-1, -1,  1); 
    \coordinate (F) at ( 0, -1,  1); 
    \coordinate (G) at ( 0,  1,  1); 
    \coordinate (H) at (-1,  1,  1); 

    % Eckpunkte der ersten schraffierten Fläche [-1,0] x [0,1] x {0} (auf der Oberseite)
%    \coordinate (S1) at (-1,  0,  0);
%    \coordinate (S2) at ( 0,  0,  0);
%    \coordinate (S3) at ( 0,  1,  0);
%    \coordinate (S4) at (-1,  1,  0);

    % Eckpunkte der zweiten schraffierten Fläche {0} x [0,1] x [-1,1] (am vorderen Rand)
    \coordinate (T1) at ( 0,  0, -1);
    \coordinate (T2) at ( 0,  1, -1);
    \coordinate (T3) at ( 0,  1,  1);
    \coordinate (T4) at ( 0,  0,  1);

    % 1. Verdeckte Kanten des Quaders (gestrichelt)
    \draw[dashed, gray, thick] (A) -- (B);
    \draw[dashed, gray, thick] (D) -- (A);
    \draw[dashed, gray, thick] (A) -- (E);

    % 2. Beide schraffierten Flächen ZUERST zeichnen
%    \draw[pattern=north east lines, pattern color=blue!50!black, mystyle, draw=blue!60!black] 
%        (S1) -- (S2) -- (S3) -- (S4) -- cycle;
    \draw[pattern=north east lines, pattern color=blue!50!black, mystyle, draw=blue!60!black] 
        (T1) -- (T2) -- (T3) -- (T4) -- cycle;

    % 3. Sichtbare Kanten des Quaders (durchgezogen)
    % Untere sichtbare Kanten
    \draw[thick] (B) -- (C) -- (D);
    % Obere Deckfläche (z = 0) - Umrandung
    \draw[thick] (E) -- (F) -- (G) -- (H) -- cycle;
    % Vertikale Kanten
    \draw[thick] (B) -- (F);
    \draw[thick] (C) -- (G);
    \draw[thick] (D) -- (H);

    % 4. Koordinatenachsen ab Ursprung O(0,0,0)
    \draw[->, thick, red!70!black] (0,0,0) -- (1.5,0,0) node[anchor=north east] {$x$ };
    \draw[->, thick, red!70!black] (0,0,0) -- (0,1.8,0) node[anchor=south west] {$y$};
    \draw[->, thick, red!70!black] (0,0,0) -- (0,0,1.5) node[anchor=south east] {$z$ };

    % Ursprung markieren
    \fill[red!70!black] (0,0,0) circle (1pt);
    \node[red!70!black] at (0,0,0) [anchor=north east] {$O$};

    % Beschriftung der Eckpunkte
%    \node[font=\footnotesize, below left] at (A) {$A$};
%    \node[font=\footnotesize, below right] at (B) {$B$};
%    \node[font=\footnotesize, above right] at (C) {$C$};
%    \node[font=\footnotesize, above left] at (D) {$D$};
%    \node[font=\footnotesize, left] at (E) {$E$};
%    \node[font=\footnotesize, right] at (F) {$F$};
%    \node[font=\footnotesize, above right] at (G) {$G$};
%    \node[font=\footnotesize, above left] at (H) {$H$};

    % Titel der Zeichnung
%    \node[above] at (0, 0, 1.5) {\textbf{};

\end{tikzpicture}
\caption{The set $\mathbf C_-$ in (ii) with the set $M= \{0\}\times {]0,1[} \times {]-1,1[}$, where Neumann conditions are imposed (hatched).  The rest has Dirichlet boundary conditions. }
\label{fig:quader5.7.ii}
\end{figure}

\end{enumerate}
\end{lemma}

\begin{proof}
We prove i), restricting to the `minus'-case.
Thanks to Theorem  \ref{t-groeger_local_global} it suffices to show \\
$A)$ 
{\it There are open, convex sets $W_{-1}, W_0, W_1,\ldots, W_n, W_\bullet$ which 
form an open covering of $\overline{{\bf Q}_-}$ and have the following additional properties:\\
{\rm a)} Each set $ W_j \cap {\bf Q}_-$ is a Lipschitz domain. \\
{\rm b)} Letting ${\bf Q}_j:= W_j \cap {\bf Q}_-$, $M_j:=M_- \cap W_j$
 and $E_j:=\partial {\bf Q}_j \setminus M_j$, the operator
\[
-\nabla \cdot \rho \nabla: W^{1,3}_{E_j}({\bf Q}_j) \to  W^{1,3}_{E_j}({\bf Q}_j)
\]
is a topological isomorphism}. \\

We first show:\\
$B) $
For $\x \in \partial {\bf Q}$ one can find a convex set
 $U_\x$ such that $U_\x \cap {\bf Q}_-$ 
results from a set $\Lambda $ in Proposition \ref{p-graphdomain} or Proposition \ref{p-01} by a bi-Lipschitzian transformation including the corresponding boundary parts. So here Proposition \ref{p-transform}  applies. In detail: Around all points from 
$ \partial {\bf Q}_- \setminus ([-1,0] \times [0,1] \times \{0 \} \cup \{0 \} \times [0,1] \times [-1,0]$ the boundary carries 
locally the Dirichlet condition. Thus, one may intersect with a suitable cube and ends up with cuboid with a pure Dirichlet condition on its boundary. This is covered by Proposition \ref{p-graphdomain}.
For $\x \in [-1,0) \times [0,1] \times \{0 \} \cup \{0 \} \times [0,1] \times [-1,0)$ one also finds a suitable prism ${\bf P}$
 such
that the intersection ${\bf P } \cap {\bf Q}_-$ is -- modulo an Euclidean movement -- one of the prisms in 
Proposition \ref{p-01}, including the corresponding boundary parts. 
\\
Thus, it remains to construct a neighborhood $ W_0= U_0$ of $0$ and neighborhoods $U_{\x}$ 
for $\x \in \{0 \} \times {]0,1]} \times \{0 \}$ which also satisfies a) and b).\\
Consider $\mathrm z_\bullet :=(-1,0, -\frac {1}{4})$, and take
 the plane $ Q$ which contains  $\mathrm z_\bullet $ and the $y$-axis. We define a bi-Lipschitzian transformation
 $\mathfrak l: \R^3 \to \R^3$ as follows: $\mathfrak l$ leaves the points which lie on $ Q$ or below $ Q$ invariant. 
On the complementary half space 
$\mathfrak l$ acts as the linear mapping $\mathfrak l_+$ which is determined as follows:
$\mathfrak l_+$ acts as the identity on $ Q$ and maps the vector $(-1,0,0)$ onto
$(0,0,1)$. 
Observe that $\{z_\bullet, (0,1,0), (-1,0,0)\}$ as well as  $\{z_\bullet, (0,1,0), (0,0,1)\}$ form %ES
bases of $\R^3$.\\
Define  $\lozengeninety$ as the open square with vertices 
\[
(-1,0), \quad (0,1), \quad (1,0), \quad (0,-1)
\]
as a subset of the $x - y$ plane and 
and $\lhd$ as the open left  half of this. Further, we put 
${\bf P}_{\lozengeninety}:=\lozengeninety \times{ ]-1,1[} \subset \R^3$ and 
${\bf P}_{\lhdbig}:=\lhd \times {]-1,1[}  \subset \R^3$. 
Then, for sufficiently small $\alpha >0$, one has
\[
\alpha {\bf P}_{\lozengeninety} \cap \mathfrak l ({\bf Q}_-) = \alpha 
{\bf P}_{\lhdbig} , \quad \text{and}
\quad \alpha {\bf P}_{\lozengeninety} \cap \mathfrak l (M_-) = \{0\} \times {]0,\alpha[} \times
{]-\alpha, \alpha[}.
\]
Letting $ W_0 :=\mathfrak l ^{-1} (\alpha {\bf P}_{\lozengeninety})$,  one obtains
for sufficiently small $\alpha$
\[
\mathfrak l(  W_0 \cap {\bf Q}_-) = \alpha {\bf P}_{\lhdbig},
 \quad \text{and} \quad
\mathfrak l( W_0 \cap M_-) =\{0\} \times { ]0,\alpha[} \times{ ]-\alpha, \alpha[}. 
\]
Moreover, it is clear that $ Q$  neither intersects the ground plate nor the cover plate
of $\alpha {\bf P}_{\lhdbig}$, and that the transformed coefficient function on 
$\alpha {\bf P}_{\lhdbig}$ is constant on both components of 
$\alpha {\bf P}_{\lhdbig} \setminus  Q$. Thus, one is, 
concerning $\Lambda:=\alpha {\bf P}_{\lhdbig}$
 and $M:=\{0\} \times {]0,\alpha[} \times {]-\alpha, \alpha[}$ again in the situation of Proposition
 \ref{p-01} (see Figure 4), and, hence, an application of Proposition \ref{p-transform} shows that 
$ W_0$ has the required properties.  For the points of $\x \in \{0 \} \times {]0,1]} \times \{0 \}$ one proceeds 
the same way, this time again ending up in Proposition \ref{p-01} (see Figure 3 now).
This proves $B)$. Chose now a finite subcovering of $\partial {\bf Q}_-$, say $W_0,W_1, \ldots W_n$ and a cuboid 
${\bf K} \subset {\bf Q}_-$ which forms, together with the sets $W_0,W_1, \ldots W_n$, an open covering of 
$\overline {\bf Q }_-$. 
Then 
\[
-\nabla \cdot \rho \nabla: W^{1,q}_0(\Omega) \to W^{-1,q}(\Omega) 
\]
also is a topological isomorphism for $q=3$, thanks to Proposition \ref{p-graphdomain}.
Then Theorem \ref{t-groeger_local_global} implies that \eqref{e-iso5} is an isomorphism for $q=3$. 
Moreover, ${\bf Q}_-$ is a Lipschitz domain, and $E_-$ is a $2$-set, see \eqref{e-twoset}.
 Hence the set of numbers $q$, for which \eqref{e-iso5} is a
 topological isomorphism, is an \emph{open interval} by Proposition \ref{e-interpolatt} that contains $2$ and $3$. \\
ii) is proved along the same lines; this time all points from $\overline {{\bf C}_\pm}$
are even covered by the model sets in Proposition \ref{p-graphdomain} and Proposition \ref{p-01}.
%iii) is obtained from ii) by means of Proposition \ref{p-transform}, there taking $\phi$ as $\iota$.
\end{proof}

\begin{corollary} \label{c-Neumannkant}
Assume that $\rho_\pm$ are \emph{constant} coefficient functions  on ${\bf Q}_\pm$, respectively. Put
\begin{equation} \label{e-M7}
M_- :={]-1,0[} \times {]0,1[} \times \{0\} 
\end{equation}

%\begin{figure}[h]
%    \centerline{\includegraphics[scale=0.9]{BuriedBild2.pdf}}
%     \caption{${\bf Q}_- $ with the grey Neumann surface $M_-$; the remaining surfaces carry Dirichlet b.c..}
%  \end{figure}
\begin{figure}[htbp]
\centering
\begin{tikzpicture}[
    scale=2.0,
    % Achsenrichtungen für Kavaliersperspektive:
    % x = Westsüdwest, skaliert mit Faktor 0,6
    % y = Osten (nach rechts)
    % z = Norden (nach oben)
    x={(-0.5543cm, -0.2296cm)}, 
    y={(1cm, 0cm)}, 
    z={(0cm, 1cm)},
    mystyle/.style={line width=1.2pt}
]
    % Eckpunkte des Quaders Qminus = [-1,1] x [-1,1] x [-1,0]
    % Untere Fläche (z = -1)
    \coordinate (A) at (-1, -1, -1); 
    \coordinate (B) at ( 0, -1, -1); 
    \coordinate (C) at ( 0,  1, -1); 
    \coordinate (D) at (-1,  1, -1); 
    
    % Obere Fläche (z = 0)
    \coordinate (E) at (-1, -1,  0); 
    \coordinate (F) at ( 0, -1,  0); 
    \coordinate (G) at ( 0,  1,  0); 
    \coordinate (H) at (-1,  1,  0); 

    % Eckpunkte der ersten schraffierten Fläche [-1,0] x [0,1] x {0} (auf der Oberseite)
    \coordinate (S1) at (-1,  0,  0);
    \coordinate (S2) at ( 0,  0,  0);
    \coordinate (S3) at ( 0,  1,  0);
    \coordinate (S4) at (-1,  1,  0);
    
     % Eckpunkte der zweiten gepunkteten Fläche {0} x [0,1] x [-1,0] (im Quaderinneren)
    \coordinate (U1) at ( 0,  -1, 0);
    \coordinate (U2) at ( 0,  0, 0);
    \coordinate (U3) at ( 0,  0,  -1);
    \coordinate (U4) at ( 0,  -1,  -1);

    % Eckpunkte der zweiten schraffierten Fläche {0} x [0,1] x [-1,0] (im Quaderinneren)
%    \coordinate (T1) at ( 0,  0, -1);
%    \coordinate (T2) at ( 0,  1, -1);
%    \coordinate (T3) at ( 0,  1,  0);
%    \coordinate (T4) at ( 0,  0,  0);

    % 1. Verdeckte Kanten des Quaders (gestrichelt)
    \draw[dashed, gray, thick] (A) -- (B);
    \draw[dashed, gray, thick] (D) -- (A);
    \draw[dashed, gray, thick] (A) -- (E);

    % 2. Beide schraffierten Flächen ZUERST zeichnen
    \draw[pattern=north east lines, pattern color=blue!50!black, mystyle, draw=blue!60!black] 
        (S1) -- (S2) -- (S3) -- (S4) -- cycle;
   \draw[pattern=dots, pattern color=blue!50!black, mystyle, draw=blue!60!black] 
        (U1) -- (U2) -- (U3) -- (U4) -- cycle;
      
%    \draw[pattern=north east lines, pattern color=blue!50!black, mystyle, draw=blue!60!black] 
%        (T1) -- (T2) -- (T3) -- (T4) -- cycle;

    % 3. Sichtbare Kanten des Quaders (durchgezogen)
    % Untere sichtbare Kanten
    \draw[thick] (B) -- (C) -- (D);
    % Obere Deckfläche (z = 0) - Umrandung
    \draw[thick] (E) -- (F) -- (G) -- (H) -- cycle;
    % Vertikale Kanten
    \draw[thick] (B) -- (F);
    \draw[thick] (C) -- (G);
    \draw[thick] (D) -- (H);

    % 4. Koordinatenachsen ab Ursprung O(0,0,0)
    \draw[->, thick, red!70!black] (0,0,0) -- (2.2,0,0) node[anchor=north east] {$x$ };
    \draw[->, thick, red!70!black] (0,0,0) -- (0,1.8,0) node[anchor=south west] {$y$};
    \draw[->, thick, red!70!black] (0,0,0) -- (0,0,0.5) node[anchor=south east] {$z$ };

    % Ursprung markieren
    \fill[red!70!black] (0,0,0) circle (1pt);
    \node[red!70!black] at (0,0,0) [anchor=north east] {$O$};

    % Beschriftung der Eckpunkte
%    \node[font=\footnotesize, below left] at (A) {$A$};
%    \node[font=\footnotesize, below right] at (B) {$B$};
%    \node[font=\footnotesize, above right] at (C) {$C$};
%    \node[font=\footnotesize, above left] at (D) {$D$};
%    \node[font=\footnotesize, left] at (E) {$E$};
%    \node[font=\footnotesize, right] at (F) {$F$};
%    \node[font=\footnotesize, above right] at (G) {$G$};
%    \node[font=\footnotesize, above left] at (H) {$H$};

    % Titel der Zeichnung
%    \node[above] at (0, 0, 1.5) {\textbf{};

\end{tikzpicture}
\caption{The set $\mathbf Q_-$ with the set $M_-={]-1,0[}\times{]0,1[}\times  \{0\}$ (hatched), where Neumann conditions are imposed.  The rest has Dirichlet boundary conditions. The dotted surface is the intersection with $\Sigma_1$.}
\label{fig:quader(5.8)}
\end{figure}

or
\begin{equation} \label{e-M8}
M_- :=({]-1,0[} \times {]-1,1[} \times \{0\}) \cup (\{0\} \times {]0,1[} \times {]-1,0]}).
\end{equation}

\begin{figure}[htbp]
\centering
\begin{tikzpicture}[
    scale=2.0,
    % Achsenrichtungen für Kavaliersperspektive:
    % x = Westsüdwest, skaliert mit Faktor 0,6
    % y = Osten (nach rechts)
    % z = Norden (nach oben)
    x={(-0.5543cm, -0.2296cm)}, 
    y={(1cm, 0cm)}, 
    z={(0cm, 1cm)},
    mystyle/.style={line width=1.2pt}
]
    % Eckpunkte des Quaders Qminus = [-1,0] x [-1,1] x [-1,0]
    % Untere Fläche (z = -1)
    \coordinate (A) at (-1, -1, -1); 
    \coordinate (B) at ( 0, -1, -1); 
    \coordinate (C) at ( 0,  1, -1); 
    \coordinate (D) at (-1,  1, -1); 
    
    % Obere Fläche (z = 0)
    \coordinate (E) at (-1, -1,  0); 
    \coordinate (F) at ( 0, -1,  0); 
    \coordinate (G) at ( 0,  1,  0); 
    \coordinate (H) at (-1,  1,  0); 

    % Eckpunkte der ersten schraffierten Fläche [-1,0] x [-1,1] x {0} (auf der Oberseite)
    \coordinate (S1) at (-1,  -1,  0);
    \coordinate (S2) at ( 0,  -1,  0);
    \coordinate (S3) at ( 0,  1,  0);
    \coordinate (S4) at (-1,  1,  0);

    % Eckpunkte der zweiten schraffierten Fläche {0} x [0,1] x [-1,0] (im Quaderinneren)
    \coordinate (T1) at ( 0,  0, -1);
    \coordinate (T2) at ( 0,  1, -1);
    \coordinate (T3) at ( 0,  1,  0);
    \coordinate (T4) at ( 0,  0,  0);
    
      % Eckpunkte der zweiten gepunkteten Fläche {0} x [0,1] x [-1,0] (im Quaderinneren)
    \coordinate (U1) at ( 0,  -1, 0);
    \coordinate (U2) at ( 0,  0, 0);
    \coordinate (U3) at ( 0,  0,  -1);
    \coordinate (U4) at ( 0,  -1,  -1);

    % 1. Verdeckte Kanten des Quaders (gestrichelt)
    \draw[dashed, gray, thick] (A) -- (B);
    \draw[dashed, gray, thick] (D) -- (A);
    \draw[dashed, gray, thick] (A) -- (E);

    % 2. Beide schraffierten Flächen ZUERST zeichnen
    \draw[pattern=north east lines, pattern color=blue!50!black, mystyle, draw=blue!60!black] 
        (S1) -- (S2) -- (S3) -- (S4) -- cycle;
    \draw[pattern=north east lines, pattern color=blue!50!black, mystyle, draw=blue!60!black] 
        (T1) -- (T2) -- (T3) -- (T4) -- cycle;
   \draw[pattern=dots, pattern color=blue!50!black, mystyle, draw=blue!60!black] 
        (U1) -- (U2) -- (U3) -- (U4) -- cycle;

    % 3. Sichtbare Kanten des Quaders (durchgezogen)
    % Untere sichtbare Kanten
    \draw[thick] (B) -- (C) -- (D);
    % Obere Deckfläche (z = 0) - Umrandung
    \draw[thick] (E) -- (F) -- (G) -- (H) -- cycle;
    % Vertikale Kanten
    \draw[thick] (B) -- (F);
    \draw[thick] (C) -- (G);
    \draw[thick] (D) -- (H);

    % 4. Koordinatenachsen ab Ursprung O(0,0,0)
    \draw[->, thick, red!70!black] (0,0,0) -- (2.2,0,0) node[anchor=north east] {$x$ };
    \draw[->, thick, red!70!black] (0,0,0) -- (0,1.8,0) node[anchor=south west] {$y$};
    \draw[->, thick, red!70!black] (0,0,0) -- (0,0,0.5) node[anchor=south east] {$z$ };

    % Ursprung markieren
    \fill[red!70!black] (0,0,0) circle (1pt);
    \node[red!70!black] at (0,0,0) [anchor=north east] {$O$};

    % Beschriftung der Eckpunkte
%    \node[font=\footnotesize, below left] at (A) {$A$};
%    \node[font=\footnotesize, below right] at (B) {$B$};
%    \node[font=\footnotesize, above right] at (C) {$C$};
%    \node[font=\footnotesize, above left] at (D) {$D$};
%    \node[font=\footnotesize, left] at (E) {$E$};
%    \node[font=\footnotesize, right] at (F) {$F$};
%    \node[font=\footnotesize, above right] at (G) {$G$};
%    \node[font=\footnotesize, above left] at (H) {$H$};

    % Titel der Zeichnung
%    \node[above] at (0, 0, 1.5) {\textbf{};

\end{tikzpicture}
\caption{The set $\mathbf Q_-$ with the set $M_-=({]-1,0[}\times{]0,1[}\times  \{0\})\cup 
(\{0\}\times {]-1,1[}\times {]-1,0[})$ (hatched),, where Neumann conditions are imposed.  
The rest has Dirichlet boundary conditions. The dotted surface is the intersection with $\Sigma_1$.}
\label{fig:quader(5.9)}
\end{figure}

%\begin{figure}[h]
%    \centerline{\includegraphics[scale=0.9]{BuriedBild4.pdf}}
%     \caption{${\bf Q}_-$ with the Neumann surface $M_-$ of \eqref{e-M8}, consisting of the grey and the dotted part. The remaining surfaces carry Dirichlet b.c.. }
%  \end{figure}

$M_+=\iota M_-$ and $E_\pm = \partial {\bf Q}_\pm \setminus M_\pm$. Then the same conclusion 
as in  Lemma $\ref{c-modifymodel}$ ii) holds.
\end{corollary}
\begin{proof}
In case of \eqref{e-M7} all sides carry a pure Dirichlet condition -- except one. On this remaining side
one half is Dirichlet and the other is Neumann. This is exactly the case which was treated in Lemma 
\ref{c-modifymodel} ii). 

In case of \eqref{e-M8} one reflects (compare \cite[Proposition 4.13]{mitRobert}) the problem across
the $x-y$-plane and again ends up with a problem as in Lemma \ref{c-modifymodel} ii).
\end{proof}

%%%%%%%%%%%%%%%%%%%%%%%%%%%%%%%%%%%%%%%%%%%%%%%%%%%%%%%%%%%%%%%%%%%%%%%%%%%%%

Now we will establish the required auxiliary results for points in $D_c^\parallel$: In  Assumption
\ref{d-rigidincl} ii) it is only supposed that the corresponding limit exists -- whatever this limit is. In the 
sequel we will modify the  $C^1$ diffeomorphism $\phi_\x:U_\x\to \mathbf C$ in \ref{d-rigidincl} i) in such a manner  that the resulting limit commutes with the linear mapping $\iota$.

\begin{lemma} \label{l-rotat}
Let $\mathfrak a$ be a symmetric, positive definite $3\times 3$ matrix. 
\renewcommand{\labelenumi}{\roman{enumi})}
Then there is linear bijection  $\mathfrak b:\R^3 \to \R^3$, mapping the $\y-\z$-plane
onto itself, such that the matrix $\frac {1}{|\det \mathfrak b|} \mathfrak b \mathfrak a \mathfrak b^T$
 is the identity. 
\end{lemma}
\begin{proof}
The assumption implies the existence of an orthogonal matrix $\mathfrak o$ such that 
$\mathfrak o\mathfrak a\mathfrak o^T=  \mathfrak {diag}(a,b,c)$ with  $a,b,c>0$. Next note that for $\mathfrak s =  \mathfrak {diag}(\sqrt{bc} , \sqrt{ac} , \sqrt{ab})$
\begin{equation} \label{e-rotatecompress}
(abc)^{-1} \mathfrak s \mathfrak o\mathfrak a\mathfrak o^T\mathfrak s^T = \text{\rm Id}.
\end{equation}
Moreover,  $\det \, \mathfrak s = \det \, \mathfrak a = abc$.
Let $H$ be the image of $\text{\rm span} \{\mathrm e_2,\mathrm e_3\}$ under 
$ \mathfrak s \mathfrak o$. 
Choose an orthonormal basis $\{\mathrm h_2,\mathrm h_3\}$ of $H$.
Let $\mathfrak v$ be an orthogonal map in $\R^3$, taking $\mathrm h_j$ to $\mathrm e_j$, $j=2,3$. 
$\mathfrak b =\mathfrak v\mathfrak s\mathfrak o$ maps 
 $\text{\rm span} \{\mathrm  e_2,\mathrm  e_3\}$ onto itself and 
\eqref{e-rotatecompress} implies
$$ \frac {1}{|\det \mathfrak b|} \mathfrak b \mathfrak a \mathfrak b^T = (abc)^{-1} \mathfrak{vsoa}
\mathfrak o^T  \mathfrak s^T\mathfrak v^T = % (abc)^{-1} \, 
\mathfrak{v}\, {\rm Id}\, \mathfrak v^T
= {\rm Id},$$ 
since $\mathfrak v^T=\mathfrak v^{-1}$.  
\end{proof}

Having this at hand, the next lemma allows to reproduce the geometric constellation in Assumption \ref{d-rigidincl}
in case $\Sigma = \Sigma _1$ \emph{and} the additional property that the limit of the resulting coefficient
function towards $0$ is a scalar multiple of ${\rm Id}_{\R^3}$. Moreover, the cases $\Sigma = \Sigma_2$ and 
$\Sigma = \Sigma_3$ are reduced in some sense to $\Sigma = \Sigma_1$.
Hereby, the resulting coefficient function has limits in ${\bf Q}$ and 
${\bf C} \setminus \overline {{\bf Q}}$ for $\mathrm z \to 0$ which are of a particularly simple form: 
They commute with $\iota$.

\begin{lemma} \label{c-transform2}
The neighborhood $U_\x$ and the diffeomorphism $\phi_\x$ in  Assumption $\ref{d-rigidincl}$ can be modified to 
$\widetilde U =\widetilde { U_\x}$, and a $C^1$-diffeomorphism  $\widetilde \phi_\x$ with the following 
properties:
\renewcommand{\labelenumi}{\roman{enumi})}
\begin{enumerate}
\item
\begin{eqnarray} \label{e-001} 
\widetilde \phi_\x(\Omega \cap \widetilde U)&=& {\bf C} \setminus \Sigma, \nonumber\\
\widetilde \phi_\x( \widetilde U \cap D^\parallel)&=&\widetilde  \phi_\x( \widetilde U \cap D_c^\parallel)=
\Sigma \cap  {\bf   C}  , \\
\widetilde  \phi_\x(\x)&=&0,\nonumber
\end{eqnarray} 
where $\Sigma$ is one of the sets defined in Assumption $\ref{d-rigidincl}$.
\item
If $\widetilde \mu$ denotes the coefficient function obtained from $\mu$ under the transformation
$ {\widetilde \phi_\x}$ (cf. Proposition $\ref{p-transform}$), then, in case  $\Sigma = \Sigma_1$, $\widetilde \mu$ 
has a limit for $ \mathrm z \to 0$ in $ {\bf C} \setminus \Sigma_1$
which equals a scalar multiple of the identity matrix.\\
In case of $\Sigma = \Sigma_2$ or $\Sigma = \Sigma_3$ the limits 
\[
\lim_{\mathrm z \in  {\bf Q}  \setminus \Sigma_1, \;  \mathrm z \to 0} \widetilde \mu(\mathrm z) \quad 
\text{and} \quad 
\lim_{\mathrm z \in  {\bf C} \setminus \overline {{\bf Q}} , \; \mathrm z \to 0} 
\widetilde \mu(\mathrm z)
\]
exist and are of the form 
\begin{equation} \label{e-matrixforum}
\left(
\begin{array}{ccc}
\beta  & 0 & 0\\
0 & a^\pm_{2,2} & a^\pm_{2,3}\\
0& a^\pm_{2,3} & a^\pm_{3,3}
\end{array}
\right), \quad \beta >0.
\end{equation}
Hence they commute with the matrix $\iota$ defined in Section $\ref{ss-prelim}$.
\end{enumerate}
\end{lemma}

\begin{proof}
Let $\underline \mu$ denote the coefficient function which is obtained from $\mu$ when transforming under 
$\phi_\x$, cf. \eqref{e-transformat1}. Since $\lim_{\z \mapsto \x, \z \in \Omega}\mu(\z)$ exists by Assumption
 \ref{d-rigidincl} and $\phi_\x$ is $C^1$
\begin{equation} \label{e-transformphi}
\lim_{ \mathrm z \to 0, 
\mathrm z \in \phi_\x( U_\x \cap \Omega)} \underline \mu(\mathrm z)=
\lim_{ \mathrm z \to 0, \,
\mathrm z \in {\bf C} \setminus \Sigma} \underline \mu(\mathrm z)
=:\mu_\star
\end{equation}
also exists. 
 Denote by $\Sigma$ one of the sets $\Sigma_1$, $\Sigma_2$ or $\Sigma_3$. By assumption 
 $\phi_\x(\Omega \cap U_\x)= {\bf C} \setminus \Sigma $. 
 We then apply the transformation $\mathfrak b$ constructed in
 Lemma \ref{l-rotat} to $\mathfrak a = \mu_\star$. 
 We set 
 \begin{eqnarray}
\label{eq.hatphi}
\widehat \phi_\x := \mathfrak b \phi_\x.
\end{eqnarray}
The essential point is, due to the construction of $\widehat \phi_\x $, that $\Sigma$ again lies in 
the $y-z$ plane.
Denoting the resulting coefficient function by $\check \mu$, we obtain 
\begin{equation} \label{e-transformB}
\lim_{ \mathrm z \to 0,\, \mathrm z \in {\bf C} \setminus \Sigma} \check \mu(\mathrm z)
=  \text{\rm Id}.
\end{equation}

Now let us focus on the individual cases $\Sigma = \Sigma_1, \Sigma_2, \Sigma_3$.
In case $\Sigma = \Sigma_1$ we are nearly finished: 

Since the map $\mathfrak b$ defined in Lemma \ref{l-rotat} preserves the $y-z$-plane, 
we can compose  $\widehat \phi_\x$ in \eqref{eq.hatphi} with a 
rotation $\mathfrak r$ around the $x$-axis such that 
$\mathfrak r \mathfrak b (\mathrm e_3)  \in \{ \lambda e_3: \lambda \neq 0\} $ and 
$\mathfrak r \mathfrak b (\Sigma_1) \subset \{\y: \y = (0,y_2,y_3) ,  y_2 \le 0 \}$. Recall here that $\Sigma_1$
is adjacent to the $z$ axis. 
The limit of the coefficient function,
 resulting under the mapping $\check \phi_\x:= 
\mathfrak r  \widehat \phi_\x $, remains the identity on $\R^3$. Observe that the local geometry of
the thus transformed object around the point $0$ is now exactly again as that of ${\bf C} \setminus \Sigma_1$.
Of course, the global geometry of $\mathfrak r \mathfrak b ( {\bf C} \setminus \Sigma_1 )$ will be different, 
but a certain restriction to a smaller set and a homethety will repair this:
Since $0$ is an inner point of the transformed cube $\mathfrak r  \mathfrak b {\bf  C}$, there is a number
 $\alpha \in {]0,1[}$, such that the cube $\alpha {\bf C} $ is contained in $\mathfrak r \mathfrak b {\bf C} $. 
One now shrinks  
the former neighborhood $U_\x$ to $\widetilde { U_\x}:=(\check  \phi_x)^{-1} \bigl (\alpha {\bf C} \bigr )$.
The mapping properties of $\mathfrak r \mathfrak b$ guarantee that
the mapping $\widetilde \phi_\x := \frac {1}{\alpha}\check \phi_\x$
 indeed satisfies \eqref{e-2}. 
Moreover, we obtain the asserted form of the limit for the transformed coefficient function.

Let us now treat the cases $\Sigma = \Sigma_2$ and $\Sigma = \Sigma _3$: 
\begin{figure}[htbp]
\centering
\begin{tikzpicture}[
    scale=1.8,
    % x-Achse um Faktor 0.5 verkürzt (Westsüdwest)
    x={(-0.554cm, -0.23cm)}, 
    y={(1cm, 0cm)}, 
    z={(0cm, 1cm)},
   mystyle/.style={line width=1.2pt}
]
    % Eckpunkte des Würfels [-1,1]^3
    \coordinate (A) at (-1, -1, -1); 
    \coordinate (B) at ( 1, -1, -1); 
    \coordinate (C) at ( 1,  1, -1); 
    \coordinate (D) at (-1,  1, -1); 
    
    \coordinate (E) at (-1, -1,  1); 
    \coordinate (F) at ( 1, -1,  1); 
    \coordinate (G) at ( 1,  1,  1); 
    \coordinate (H) at (-1,  1,  1); 

    % Eckpunkte des schraffierten Dreieicks  
    \coordinate (S1) at (0, 0, 0);
    \coordinate (S2) at (0,  -1,-1);
    \coordinate (S3) at (0,  -1,1);
%    \coordinate (S4) at (0, -0.5, 1);

    % 1. Schraffierte innere Fläche ZUERST zeichnen (transparent ohne Hintergrundfüllung)
    \draw[pattern=north east lines, pattern color=blue!50!black, dashed, draw=blue!60!black] 
        (S3) -- (S2) -- (S1) -- cycle;

    % 2. Gestrichelte Kanten nach Vorgabe: (AB), (AD), (AE)
    \draw[dashed, gray, thick] (A) -- (B);
    \draw[dashed, gray, thick] (A) -- (D);
    \draw[dashed, gray, thick] (A) -- (E);

    % 3. Durchgezogene Kanten nach Vorgabe: (BC), (BF), (CG), (CD), (DH), (EF), (EH), (FG), (GH)
    \draw[thick] (B) -- (C);
    \draw[thick] (B) -- (F);
    \draw[thick] (C) -- (G);
    \draw[thick] (C) -- (D);
    \draw[thick] (D) -- (H);
    \draw[thick] (E) -- (F);
    \draw[thick] (E) -- (H);
    \draw[thick] (F) -- (G);
    \draw[thick] (G) -- (H);
 
    \draw[thick] (S1) -- (S2);
    \draw[thick] (S1) -- (S3);

    % 4. Koordinatenachsen ab Ursprung (0,0,0)
    \draw[->, thick, red!70!black] (0,0,0) -- (2,0,0) node[anchor=north east] {$x$};
    \draw[->, thick, red!70!black] (0,0,0) -- (0,1.8,0) node[anchor=south west] {$y$};
    \draw[->, thick, red!70!black] (0,0,0) -- (0,0,1.6) node[anchor=south east] {$z$};

    % Ursprung markieren
    \fill[red!70!black] (0,0,0) circle (1pt);
    \node[red!70!black] at (0,0,0) [anchor=north] {$O$};

    % 5. Beschriftung der Eckpunkte des Würfels
%    \node[font=\footnotesize, below left] at (A) {$A$};
%    \node[font=\footnotesize, below right] at (B) {$B$};
%    \node[font=\footnotesize, above right] at (C) {$C$};
%    \node[font=\footnotesize, above left] at (D) {$D$};
%    \node[font=\footnotesize, left] at (E) {$E$};
%    \node[font=\footnotesize, right] at (F) {$F$};
%    \node[font=\footnotesize, above right] at (G) {$G$};
%    \node[font=\footnotesize, above left] at (H) {$H$};

    % Beschriftung der schraffierten Fläche (mit Führungslinie)
    \draw[thin, blue!60!black] (0, -0.7, 0.5) -- (-0.8, 0.4, 0.5) 
        node[right, align=left, font=\small] {$\Sigma_2$};

    % Titel
 %   \node[above] at (0, 0, 1.6) {\textbf{Würfel $[-1,1]^3$ (mit verkürzter $x$-Achse)}};

\end{tikzpicture}
\caption{The cube $\mathbf C= [-1,1]^3$ without $\Sigma_2$ (hatched)}
\label{fig:Csigma2}
\end{figure}

\begin{figure}[htbp]
\centering
\begin{tikzpicture}[
    scale=1.8,
    % x-Achse um Faktor 0.5 verkürzt (Westsüdwest)
    x={(-0.554cm, -0.23cm)}, 
    y={(1cm, 0cm)}, 
    z={(0cm, 1cm)},
   mystyle/.style={line width=1.2pt}
]
    % Eckpunkte des Würfels [-1,1]^3
    \coordinate (A) at (-1, -1, -1); 
    \coordinate (B) at ( 1, -1, -1); 
    \coordinate (C) at ( 1,  1, -1); 
    \coordinate (D) at (-1,  1, -1); 
    
    \coordinate (E) at (-1, -1,  1); 
    \coordinate (F) at ( 1, -1,  1); 
    \coordinate (G) at ( 1,  1,  1); 
    \coordinate (H) at (-1,  1,  1); 

    % Eckpunkte des schraffierten Dreieicks  
    \coordinate (S1) at (0, 0, 0);
    \coordinate (S2) at (0,  -1,-1);
    \coordinate (S3) at (0, 1, -1);
    \coordinate (S4) at (0,  1,1);
    \coordinate (S5) at (0,  -1,1);
%    \coordinate (S4) at (0, -0.5, 1);

    % 1. Schraffierte innere Fläche ZUERST zeichnen (transparent ohne Hintergrundfüllung)
    \draw[pattern=north east lines, pattern color=blue!50!black, dashed, draw=blue!60!black] 
        (S1) -- (S2) -- (S3) -- (S4) -- (S5) -- cycle;

    % 2. Gestrichelte Kanten nach Vorgabe: (AB), (AD), (AE)
    \draw[dashed, gray, thick] (A) -- (B);
    \draw[dashed, gray, thick] (A) -- (D);
    \draw[dashed, gray, thick] (A) -- (E);

    % 3. Durchgezogene Kanten nach Vorgabe: (BC), (BF), (CG), (CD), (DH), (EF), (EH), (FG), (GH)
    \draw[thick] (B) -- (C);
    \draw[thick] (B) -- (F);
    \draw[thick] (C) -- (G);
    \draw[thick] (C) -- (D);
    \draw[thick] (D) -- (H);
    \draw[thick] (E) -- (F);
    \draw[thick] (E) -- (H);
    \draw[thick] (F) -- (G);
    \draw[thick] (G) -- (H);
 
    \draw[thick] (S1) -- (S2);
    \draw[thick] (S1) -- (S5);

    % 4. Koordinatenachsen ab Ursprung (0,0,0)
    \draw[->, thick, red!70!black] (0,0,0) -- (2,0,0) node[anchor=north east] {$x$};
    \draw[->, thick, red!70!black] (0,0,0) -- (0,1.8,0) node[anchor=south west] {$y$};
    \draw[->, thick, red!70!black] (0,0,0) -- (0,0,1.6) node[anchor=south east] {$z$};

    % Ursprung markieren
    \fill[red!70!black] (0,0,0) circle (1pt);
    \node[red!70!black] at (0,0,0) [anchor=north] {$O$};

    % 5. Beschriftung der Eckpunkte des Würfels
%    \node[font=\footnotesize, below left] at (A) {$A$};
%    \node[font=\footnotesize, below right] at (B) {$B$};
%    \node[font=\footnotesize, above right] at (C) {$C$};
%    \node[font=\footnotesize, above left] at (D) {$D$};
%    \node[font=\footnotesize, left] at (E) {$E$};
%    \node[font=\footnotesize, right] at (F) {$F$};
%    \node[font=\footnotesize, above right] at (G) {$G$};
%    \node[font=\footnotesize, above left] at (H) {$H$};

    % Beschriftung der schraffierten Fläche (mit Führungslinie)
    \draw[thin, blue!60!black] (0, 0.6, 0.5) -- (-0.8, 0.7, 0.5) 
        node[right, align=left, font=\small] {$\Sigma_3$};

    % Titel
 %   \node[above] at (0, 0, 1.6) {\textbf{Würfel $[-1,1]^3$ (mit verkürzter $x$-Achse)}};

\end{tikzpicture}
\caption{The cube $\mathbf C= [-1,1]^3$ without $\Sigma_3$ (hatched)}
\label{fig:CSigma3}
\end{figure}

Since $\mathfrak b$ is a linear mapping,
preserving the $y - z$ plane, $\mathfrak b(\Sigma_2)$ must be a triangle $\mathcal T$ in the $y-z$ plane 
with one vertex in $0 \in \R^3$. 
One then performs a rotation $\mathfrak r$ of the $y-z$ plane around the $x$-axis such that
\[
{\bf H}_2^- \supset \mathfrak r \mathcal T \cap {\bf H }^+_3 \neq \emptyset  \neq \mathfrak r \mathcal T 
\cap {\bf H}^-_3 \subset {\bf H}_2^- . 
\]
Obviously, $\mathfrak r \mathcal T$ remains a subset of the $y-z$-plane.  It is clear that also after this
new transformation the limit of the resulting coefficient function towards $0$ remains the identity matrix. 
Let $ \z_-$ be the unit vector along the side of $\mathfrak r \mathcal T$  which is adjacent to $0 \in \R^3$
and situated in ${\bf H}_3^-$ and $\z_+$ the unit vector along the other side of
 $\mathfrak r \mathcal T$ which is adjacent to $0 \in \R^3$. In this notation, we define
now a bi-Lipschitzian mapping $\omega:\R^3 \to \R^3$ as follows:
Within the half space $\overline {{\bf H}^-_3}$ we set $\omega$ as the linear mapping which leaves 
the $x-y$ plane fixed and maps $\z_-$ onto $- \mathrm e_3$. On the half space $\overline { {\bf H}_3^+}$
 we define $\omega$ by also  leaving
 the $x-y$ plane fixed and mapping $ \z_+$ onto $ \mathrm e_3$. This transforms $\mathfrak r\mathcal T$
 into a triangle of which one side is the segment between the vertices $(0,0,-1)$ and $(0,0,1)$ and the
 third vertex lies in the $x-y$ plane. Now one is -- locally around $0 \in \R^3$ --  in the same situation as in case $\Sigma = \Sigma_1$, as far as the geometry is concerned.
Observe that both transformations, $\mathfrak r$ and $\omega$ in fact only took place in 
the $y-z$-plane and left the $x$-direction invariant. Hence, it is clear that then a limit of the derived 
coefficient function in the half spaces ${\bf H}^\pm_3$ exists and is of the  form as postulated in 
\eqref{e-matrixforum}.
Shrinking the neighborhood of $\x$ as in the case $\Sigma = \Sigma_1$ and applying the homothety 
$\R^3 \ni \z \mapsto \frac {1}{\alpha} \z$, one has proved the assertion in case $\Sigma = \Sigma_2$.\\
 If $\Sigma= \Sigma _3$, one proceeds the same way  -- beginning with the observation that 
$\mathfrak b\bigl ( ({\bf C } \cap \{\z : z_1 =0 \}) \setminus \Sigma_3 \bigr)$ also is a triangle in the 
$y-z$-plane. At the end one applies the mapping $\mathfrak {diag}(\frac {1}{\alpha},-\frac {1}{\alpha},\frac {1}{\alpha})$.
\end{proof}

%\begin{color}{red} nachfolgender Subtitel geaendert, ist es gut?\end{color}
\subsection{The final step of the proof}\label{sec:final}
Now we present the heart of the proof for Theorem \ref{t-main}. 
We start by considering the regularity near points in $D^\parallel_d$ and $R_d$, since this is the easier situation. 
Treating points in $D^\parallel_c$ and $R_c$ requires the analysis of one more model situation, see Lemma \ref{l-multiple1}, below.  

\subsubsection{Regularity near points in $D^\parallel_d$ and $R_d$}
%{\color{magenta} 
Let $\x\in D^\parallel_d \cup R_d$. 
According to Assumption $\ref{a-00}$ and Assumption $\ref{a-01}$, respectively, there exists a neighborhood $U_\x$ of $\x$ in $\Omega\cup D$, not intersecting $\partial \Omega\setminus D$ such that $U_\x\cap \Omega$ is the union of finitely many Lipschitz domains. 
This permits the \emph{separate} consideration of the equation on the connected components of $\Omega \cap U_{\x}$.
Again Assumption  \ref{a-00} and Assumption  \ref{a-01} assure that the solutions on the separate sets do admit the required
$W^{1,q}$-regularity with a $q >3$. Finally, thanks to Lemma \ref{l-0x1}, the function on the whole set $U_\x \cap \Omega$ belongs to $W^{1,q}_0(U_\x  \cap \Omega)$ if $\x \in D^\parallel_d $ and to $W^{1,q}_D(U_\x  \cap \Omega)$ if 
$\x \in R_d$. 
%}

%{\color{blue}ALT 
%
%We first  localize the problem around any point $\x  \in D^\parallel_d \cup R_d$,
%according to Lemma \ref{l-cutoff}. Then Assumption  $\ref{a-00}$ and Assumption  $\ref{a-01}$ allow to apply Lemma \ref{l-031}. This permits the \emph{separate} consideration of the equation on the connected components of $\Omega \cap U_{\x}$.
%Again Assumption  \ref{a-00} and Assumption  \ref{a-01} assure that the solutions on the separate sets do admit the required
%$W^{1,q}$-regularity with a $q >3$. Finally, thanks to Lemma \ref{l-0x1}, the function on the whole set $U_\x \cap \Omega$ belongs to $W^{1,q}_0(U_\x  \cap \Omega)$ if $\x \in D^\parallel_d $ and to $W^{1,q}_D(U_\x  \cap \Omega)$ if 
%$\x \in R_d$.}

\subsubsection{Regularity near points in $D^\parallel_c$ and $R_c$}
The strategy which applies to both of the remaining cases $\x  \in D_c^\parallel$ and
$\x \in R_c$, respectively, is as follows.  
First -- as above -- one localizes the problem around the point $\x$
 under consideration. Here the sets $ \widetilde U_\x$ from Lemma \ref{c-transform2}
and $ { U_\x}$ from Assumption  \ref{d-Rout}, respectively, play the role of $ V$ in Lemma \ref{l-cutoff}. Afterwards one transforms the problem under the mapping $\widetilde \phi_\x $ and $\phi_\x$, respectively,  onto the corresponding model sets, thereby preserving
the quality of the problems, thanks to Proposition \ref{p-transform}. Then one `compares' the resulting problem
 with one for which the coefficient function is piecewise constant in the spirit of Lemma \ref{l-localize}.

Let us start by proving a regularity theorem for the corresponding model sets.
Here $M$ always denotes the Neumann boundary part.
\begin{lemma} \label{l-multiple1}
Let $\Lambda$ be one of the domains  
${\bf C } \setminus \Sigma_1, {\bf Q} \setminus \Sigma_1$, where $\Sigma_1$ is defined in \eqref{e-halbflaeche}.
In case  $\Lambda ={\bf C } \setminus \Sigma_1$ let $M =\emptyset$.
In case $\Lambda = {\bf Q} \setminus \Sigma_1$, let either $M=\emptyset$
or 
\begin{equation} \label{e-M}
M =\bigl ({ ]-1,1[} \times {]-1,1[} \times \{0\}\bigr ) \setminus \Sigma_1
, \quad \text {or} \quad M = {]-1,1[} \times{ ]0,1[} \times \{0\}.
\end{equation}
In any case we set $E:=\partial \Lambda \setminus M$.

\begin{figure}[htbp]
\centering
\begin{tikzpicture}[
    scale=2.0,
    % Achsenrichtungen für Kavaliersperspektive:
    % x = Westsüdwest, skaliert mit Faktor 0,6
    % y = Osten (nach rechts)
    % z = Norden (nach oben)
    x={(-0.5543cm, -0.2296cm)}, 
    y={(1cm, 0cm)}, 
    z={(0cm, 1cm)},
    mystyle/.style={line width=1.2pt}
]
    % Eckpunkte des Quaders Q = [-1,1] x [-1,1] x [-1,0]
    % Untere Fläche (z = -1)
    \coordinate (A) at (-1, -1, -1); 
    \coordinate (B) at ( 1, -1, -1); 
    \coordinate (C) at ( 1,  1, -1); 
    \coordinate (D) at (-1,  1, -1); 
    
    % Obere Fläche (z = 1)
    \coordinate (E) at (-1, -1,  0); 
    \coordinate (F) at ( 1, -1,  0); 
    \coordinate (G) at ( 1,  1,  0); 
    \coordinate (H) at (-1,  1,  0); 

    % Eckpunkte der ersten schraffierten Fläche [-1,1] x [-1,1] x {0} (auf der Oberseite)
    \coordinate (S1) at (-1,  -1,  0);
    \coordinate (S2) at ( 1,  -1,  0);
    \coordinate (S3) at ( 1,  1,  0);
    \coordinate (S4) at (-1,  1,  0);

    % Eckpunkte der zweiten gepunkteten Fläche Sigma1= {0} x [-1,0] x [-1,0] 
    \coordinate (T1) at ( 0,  -1, -1);
    \coordinate (T2) at ( 0,  0, -1);
    \coordinate (T3) at ( 0,  0,  0);
    \coordinate (T4) at ( 0,  -1,  0);

    % 1. Verdeckte Kanten des Quaders (gestrichelt)
    \draw[dashed, gray, thick] (A) -- (B);
    \draw[dashed, gray, thick] (D) -- (A);
    \draw[dashed, gray, thick] (A) -- (E);

    % 2. Beide schraffierten Flächen ZUERST zeichnen
    \draw[pattern=north east lines, pattern color=blue!50!black, mystyle, draw=blue!60!black] 
        (S1) -- (S2) -- (S3) -- (S4) -- cycle;
    \draw[pattern=dots, pattern color=black, dashed, draw=black] 
        (T1) -- (T2) -- (T3) -- (T4) -- cycle;
   \draw[thick] (T2) -- (T3) -- (T4);
    % 3. Sichtbare Kanten des Quaders (durchgezogen)
    % Untere sichtbare Kanten
    \draw[thick] (B) -- (C) -- (D);
    % Obere Deckfläche (z = 0) - Umrandung
    \draw[thick] (E) -- (F) -- (G) -- (H) -- cycle;
    % Vertikale Kanten
    \draw[thick] (B) -- (F);
    \draw[thick] (C) -- (G);
    \draw[thick] (D) -- (H);

    % 4. Koordinatenachsen ab Ursprung O(0,0,0)
    \draw[->, thick, red!70!black] (0,0,0) -- (2,0,0) node[anchor=north east] {$x$ };
    \draw[->, thick, red!70!black] (0,0,0) -- (0,1.7,0) node[anchor=south west] {$y$};
    \draw[->, thick, red!70!black] (0,0,0) -- (0,0,0.8) node[anchor=south east] {$z$ };

    % Ursprung markieren
    \fill[red!70!black] (0,0,0) circle (1pt);
    \node[red!70!black] at (0,0,0) [anchor=north east] {$O$};

    % Beschriftung der Eckpunkte
%    \node[font=\footnotesize, below left] at (A) {$A$};
%    \node[font=\footnotesize, below right] at (B) {$B$};
%    \node[font=\footnotesize, above right] at (C) {$C$};
%    \node[font=\footnotesize, above left] at (D) {$D$};
%    \node[font=\footnotesize, left] at (E) {$E$};
%    \node[font=\footnotesize, right] at (F) {$F$};
%    \node[font=\footnotesize, above right] at (G) {$G$};
%    \node[font=\footnotesize, above left] at (H) {$H$};

    % Titel der Zeichnung
%    \node[above] at (0, 0, 1.5) {\textbf{};
  \draw[thin,black] (0, -0.5, -0.5) -- (0.8, -0.4, 0.5) 
        node[left, align=left, font=\small] {$\Sigma_1 \cap \overline{\mathbf Q}=\{0\} \times {(-1,0]} \times {(-1,0]}$};

\end{tikzpicture}
\caption{The set $\mathbf Q\setminus \Sigma_1$ with the set $M=  ({]-1,1[} \times {]-1,1[}\times \{0\})\setminus (\{0\}\times [-1,0]\times \{0\})$, where Neumann conditions are imposed (hatched).  
The rest (including $\Sigma_1$ (dotted)) has Dirichlet boundary conditions. }
\label{fig:Lemma3.12}
\end{figure}

Let $\varrho$ be an elliptic coefficient function on $\Lambda$, which satisfies 
the invariance property 
\begin{equation} \label{e-invarian}
 \varrho(\iota(\cdot )) =\iota \varrho(\cdot) \iota.
\end{equation}
Moreover, if $\Lambda = {\bf Q} \setminus \Sigma_1$, then let $\rho$ be constant on
${\bf Q}_+$, and if $\Lambda = {\bf C} \setminus \Sigma_1$, then  let
$\rho$ be constant on the sets ${\bf C}_+ \cap {\bf H}_3^\pm$.
 \\
Then, in any case, there is a $p>3$ such that 
$-\nabla \cdot \varrho \nabla:W^{1,q}_E(\Lambda) \to W_E^{-1,q}(\Lambda)$
is a topological isomorphism if $q \in [2,p[$.
\end{lemma}

\begin{proof}
Let us first note that the invariance property $ \varrho(\iota(\cdot )) =\iota \varrho(\cdot) \iota$
implies that $\rho$ is also constant on ${\bf Q}_-$, if $\Lambda = {\bf Q } \setminus \Sigma_1$, and
 $\rho$ is also constant on the sets ${\bf C}_- \cap {\bf H}_3^\pm$, if $\Lambda = {\bf C} \setminus \Sigma_1$.

Define $\Psi:L^2(\Lambda) \to L^2(\Lambda)$ by $( {\Psi}w) (\mathrm z)
=w(\iota(\z))$. 
Since $\Lambda$, $M$ and $\Sigma_1$ are invariant under $\iota$, 
$\Psi$ induces topological isomorphisms
$$ {\Psi_q}:W^{1,q}_E(\Lambda) \to W^{1,q}_E(\Lambda)$$  
for all $q \in [1,\infty[$, cf.
 Proposition \ref {p-transform}. 
 Define furthermore 
 $$ \Psi_2^*:W_E^{-1,2}(\Lambda) \to W_E^{-1,2}(\Lambda)$$ 
 as the adjoint of 
$ \Psi_2:W^{1,2}_E(\Lambda) \to W^{1,2}_E(\Lambda) $.
Assume that $u \in W^{1,2}_E(\Lambda)$ is a solution of $-\nabla \cdot \varrho \nabla u=f \in  
 W_E^{-1,2}(\Lambda)$. Then one may apply Proposition \ref {p-transform} for $\phi=\iota$. It is exactly at this point,
where the matrix equality $\iota \varrho(\iota(\cdot)) \iota =\varrho(\cdot)$ comes into play, yielding $-\nabla \cdot 
\varrho \nabla  \Psi_2 u =  \Psi_2^* f$. This implies 
\begin{equation} \label{e-gleich1}
-\nabla \cdot \varrho \nabla (u - \Psi_2 u)=f - \Psi_2^* f,
\end{equation}
and
\begin{equation} \label{e-gleich2}
-\nabla \cdot \varrho \nabla (u +  \Psi_2 u)=f +  \Psi_2^*  f.
\end{equation}

First consider \eqref{e-gleich1}. It is clear that the function $u -  \Psi_2 u$ has trace $0$ on the set 
\begin{equation} \label{e-xi}
\Xi:=\overline \Lambda \cap \{\y: y_1=0  \}.
\end{equation}
Denote the restriction of $u - \Psi_2 u$ to the sets $\Lambda_\pm :=\Lambda \cap
{\bf H}_1^\pm$ by $v_\pm$. We define 
\[
E_\pm:= (E \cap {\bf H}_1^\pm) \cup ( \partial \Lambda_\pm \cap \{\mathrm z:z_1 =0\}).
\]
Since the two domains $\Lambda_\pm$ are separated by the Dirichlet boundary part of both sets, the trace-zero 
condition on $\Xi$ allows to apply Lemma \ref{l-031}. Thus, one ends up with \emph{separate} equations on both sets
$\Lambda_\pm$.
Assume from now  on that $q \in [2,\infty[$ and $f \in W^{-1,q}_E(\Lambda)$. This implies $ f_\pm 
\in  W^{-1,q}_{E_\pm}(\Lambda_\pm)$, see \eqref{e-deffj}. 

In case  $\Lambda = {\bf  C}  \setminus \Sigma_1$ one has 
then $E_\pm =\partial \Lambda _\pm$,  see Figure 10. 

\begin{figure}[htbp]
\centering
\begin{tikzpicture}[
    scale=1.8,
    % x-Achse um Faktor 0.5 verkürzt (Westsüdwest)
    x={(-0.554cm, -0.23cm)}, 
    y={(1cm, 0cm)}, 
    z={(0cm, 1cm)},
   mystyle/.style={line width=1.2pt}
]
    % Eckpunkte des Würfels [-1,1]^3
    \coordinate (A) at (-1, -1, -1); 
    \coordinate (B) at ( 0, -1, -1); 
    \coordinate (C) at ( 0,  1, -1); 
    \coordinate (D) at (-1,  1, -1); 
    
    \coordinate (E) at (-1, -1,  1); 
    \coordinate (F) at ( 0, -1,  1); 
    \coordinate (G) at ( 0,  1,  1); 
    \coordinate (H) at (-1,  1,  1); 

    % Eckpunkte der punktierten Fläche {0} x [-1/2, 1/2] x [0, 1]
    \coordinate (S1) at (0, -1, 1);
    \coordinate (S2) at (0,  0, 1);
    \coordinate (S3) at (0,  0, -1);
    \coordinate (S4) at (0, -1, -1);

    % Eckpunkte der punktierten Fläche {0} x [-1/2, 1/2] x [0, 1]
    \coordinate (T1) at (0, 0, 1);
    \coordinate (T2) at (0,  1, 1);
    \coordinate (T3) at (0,  1, -1);
    \coordinate (T4) at (0, 0, -1);

    % 1. Punktierte Fläche 
    \draw[pattern=dots, pattern color=blue!50!black, dashed, draw=blue!60!black] 
        (S3) -- (S2) -- (S1) -- (S4) -- cycle;

    % 1. Schraffierte innere Fläche ZUERST zeichnen (transparent ohne Hintergrundfüllung)
    \draw[pattern=north east lines, pattern color=blue!50!black, thick, draw=blue!60!black] 
        (T1) -- (T2) -- (T3) -- (T4) -- cycle;

    % 2. Gestrichelte Kanten nach Vorgabe: (AB), (AD), (AE)
    \draw[dashed, gray, thick] (A) -- (B);
    \draw[dashed, gray, thick] (A) -- (D);
    \draw[dashed, gray, thick] (A) -- (E);

    % 3. Durchgezogene Kanten nach Vorgabe: (BC), (BF), (CG), (CD), (DH), (EF), (EH), (FG), (GH)
    \draw[thick] (B) -- (C);
    \draw[thick] (B) -- (F);
    \draw[thick] (C) -- (G);
    \draw[thick] (C) -- (D);
    \draw[thick] (D) -- (H);
    \draw[thick] (E) -- (F);
    \draw[thick] (E) -- (H);
    \draw[thick] (F) -- (G);
    \draw[thick] (G) -- (H);
   % \draw[dashed,gray,very thick] (S3) -- (S4);

    % 4. Koordinatenachsen ab Ursprung (0,0,0)
    \draw[->, thick, red!70!black] (0,0,0) -- (2,0,0) node[anchor=north east] {$x$};
    \draw[->, thick, red!70!black] (0,0,0) -- (0,1.8,0) node[anchor=south west] {$y$};
    \draw[->, thick, red!70!black] (0,0,0) -- (0,0,1.4) node[anchor=south east] {$z$};

    % Ursprung markieren
    \fill[red!70!black] (0,0,0) circle (1pt);
    \node[red!70!black] at (0,0,0) [anchor=north] {$O$};

    % 5. Beschriftung der Eckpunkte des Würfels
%    \node[font=\footnotesize, below left] at (A) {$A$};
%    \node[font=\footnotesize, below right] at (B) {$B$};
%    \node[font=\footnotesize, above right] at (C) {$C$};
%    \node[font=\footnotesize, above left] at (D) {$D$};
%    \node[font=\footnotesize, left] at (E) {$E$};
%    \node[font=\footnotesize, right] at (F) {$F$};
%    \node[font=\footnotesize, above right] at (G) {$G$};
%    \node[font=\footnotesize, above left] at (H) {$H$};

    % Beschriftung der schraffierten Fläche (mit Führungslinie)
%    \draw[thin, blue!60!black] (0, 0.2, 0.5) -- (-0.8, 0.4, 0.5) 
%        node[right, align=left, font=\small] {$D^\parallel =\{0\} \times [-\frac{1}{2},\frac{1}{2}] \times [0,1)$};

    % Titel
 %   \node[above] at (0, 0, 1.6) {\textbf{Würfel $[-1,1]^3$ (mit verkürzter $x$-Achse)}};

\end{tikzpicture}
\caption{The figure shows the  half-cube $[-1,0]\times [-1,1]^2$. The whole boundary is Dirichlet. 
The figure also shows  the  surfaces $\Sigma_1$ (dotted),  and  $\{0\}\times [0,1]\times [-1,1]$ (hatched).
On the latter, the Dirichlet condition arises from antisymmetric reflection.
}
\label{fig:4.11.1}
\end{figure}
%\begin{figure}[h]
%    \centerline{\includegraphics[scale=0.9]{BuriedBild6.pdf}}
%     \caption{The case $\Lambda = {\bf C} \setminus \Sigma_1$, $M=\emptyset$. The figure shows  ${\bf C}_-$  with the surface $\Sigma_1$ from $D^\parallel$  in black. 
%The whole boundary is Dirichlet, on the crosshatched part due to antisymmetric reflection. 
%The case  $\Lambda= {\bf Q }\setminus \Sigma_1$, $M=\emptyset$ gives an analogous picture for ${\bf Q}_-$.  }
%  \end{figure}

 Hence, one may apply Proposition \ref{p-convexhetero}. It tells
us that,   for all $q \in [2,p[$, $v_\pm \in W^{1,q}_0(\Lambda_\pm)$, if $f \in W^{-1,q}(\Lambda_\pm)$ for a certain
$p >3$. Since the traces of $v_\pm$ coincide on the common frontier of $\Lambda_\pm$, this even gives
$u - \Psi_2 u \in W^{1,q}_0(\Lambda)$ for the same range of $q$'s, cf. Lemma \ref{l-0x1}.
If $\Lambda =\mathfrak Q \setminus \Sigma_1$ and $M$ is empty, one concludes in the same way.

Consider now the case $\Lambda ={\bf  Q }\setminus \Sigma_1$ with $M$ being one of the sets
in \eqref{e-M}. In principle,  one can argue here as before -- with one difference: The occurring Neumann boundary parts 
$M_\pm:=M \cap {\bf  H}_1^\pm$ are nontrivial. Nevertheless, if defining $E_\pm :=
\partial \Lambda_\pm \setminus M_\pm$ and restricting the problem
to each of the two components $\Lambda_\pm$, then the resulting setting fits into
Corollary \ref{e-corspiegel}, if $M$ is the first set in \eqref{e-M},  and into 
Corollary \ref{c-Neumannkant} if $M$ equals the second set in \eqref{e-M}.
This gives $u- \Psi_2u \in W^{1,q}_{E_\pm }(\Lambda_\pm)$ for $q \in [2,p[$, with $p >3$.
Since both traces on the common interface $\partial \Lambda_+ \cap \partial \Lambda_-$ are zero, 
$u -\Psi_2u$ even belongs to $W^{1,q}_{E}(\Lambda)$ for this same range of $q$'s, thanks to 
Lemma \ref{l-0x1}. 

Next consider  \eqref{e-gleich2}, first assuming $f \in L^2(\Lambda)
\hookrightarrow W^{-1,q}_E(\Lambda)$ for $q \in [2,6]$. Then %$ \Psi_2^* f=\Psi f$, and 
both $v:=u +  \Psi_2 u$ and $g:=f + \Psi_2^* f$ are invariant under $\Psi_2$, $\Psi_2^*$, respectively.
Let us establish a corresponding equation for $v |_{\Lambda_-}$:  Let 
\begin{equation} \label{e-posssibilities}
M_-= (M \cap \overline {{\bf  H}_1^-}) \cup  \{0 \} \times{ ]0,1[} \times {]-1,0]}
 \subset \partial \Lambda _-, \  E_-:=\partial \Lambda_- \setminus M_-.
\end{equation}
Note that $M_-$ is then \emph{open} in 
$\partial \Lambda_-$.
For $w \in W^{1,r}_{E_-}(\Lambda_-)$ ($r \in ]1,\infty [$) we define
$\widehat w$ on $\Lambda$ by 
\begin{equation} \label{e-defreflect}
\widehat w(\y) = \begin{cases} 
w(\y), & \y \in \Lambda_-\\
w(\iota(\y)),&\y \in \Lambda_+\\
{\rm tr} \,w,  &\y\in \Lambda \cap \{\mathrm z:z_1=0 \},
\end{cases}
\end{equation}
where $\rm tr$ denotes the corresponding trace.
One knows that $\widehat w \in W^{1,r}_E(\Lambda)$, if $w \in W^{1,r}_{E_-}(\Lambda_-)$, cf. 
\cite[Proposition  4.4]{TomJo}. Now Prop. \ref{p-transform}  in combination with the condition \eqref{e-invarian} 
gives $\int_{\Lambda_-}\varrho \nabla v \cdot \nabla \widehat w = \int_{\Lambda_+}\varrho \nabla v \cdot \nabla 
\widehat w$ since both, $v$ and $\widehat w$ are invariant under $\iota$. This leads to
%Thus, \eqref{e-gleich2} in combination with \eqref{e-defreflect} yields
\begin{equation} \label{e-weakeq}
2 \int_{\Lambda_-}\varrho \nabla v \cdot \nabla w \, d\y = \int _\Lambda \varrho 
\nabla v \cdot \nabla \widehat w \, d\y \\
\int_\Lambda g \widehat w \, dx = 2 \int_{\Lambda_-} g w \,
 d\y. % \nonumber
\end{equation}
Moreover, Lemma \ref{l-001} tells us that $v \in W^{1,2}_E(\Lambda)$ implies $v|_{\Lambda_-} 
\in W^{1,2}_{E_-}(\Lambda_-)$. Consequently, \eqref{e-weakeq} can be interpreted as the weak formulation of the equation 
\[
-\nabla \cdot \varrho \nabla  v|_{\Lambda_-}=g|_{\Lambda_-}.
\]
Concerning the constellations admitted in the assumptions, \eqref{e-posssibilities} allows for the following possibilities:
\begin{itemize}
\item If $M=\emptyset$:
\begin{equation} \label{e-possi1}
\Lambda_- = {\bf  C }_-, \quad M_-= \{0\} \times ]0,1[ \times ]-1,1[,
\end{equation}
\begin{equation} \label{e-possi2}
\Lambda_- = {\bf Q}_-, \quad M_-= \{0\} \times ]0,1[ \times ]-1,0[
\end{equation}

\begin{figure}[htbp]
\centering
\begin{tikzpicture}[
    scale=1.8,
    % x-Achse um Faktor 0.5 verkürzt (Westsüdwest)
    x={(-0.554cm, -0.23cm)}, 
    y={(1cm, 0cm)}, 
    z={(0cm, 1cm)},
   mystyle/.style={line width=1.2pt}
]
    % Eckpunkte des Würfels [-1,1]^3
    \coordinate (A) at (-1, -1, -1); 
    \coordinate (B) at ( 0, -1, -1); 
    \coordinate (C) at ( 0,  1, -1); 
    \coordinate (D) at (-1,  1, -1); 
    
    \coordinate (E) at (-1, -1,  1); 
    \coordinate (F) at ( 0, -1,  1); 
    \coordinate (G) at ( 0,  1,  1); 
    \coordinate (H) at (-1,  1,  1); 

    % Eckpunkte der punktierten Fläche {0} x [-1/2, 1/2] x [0, 1]
    \coordinate (S1) at (0, -1, 1);
    \coordinate (S2) at (0,  0, 1);
    \coordinate (S3) at (0,  0, -1);
    \coordinate (S4) at (0, -1, -1);

    % Eckpunkte der punktierten Fläche {0} x [-1/2, 1/2] x [0, 1]
    \coordinate (T1) at (0, 0, 1);
    \coordinate (T2) at (0,  1, 1);
    \coordinate (T3) at (0,  1, -1);
    \coordinate (T4) at (0, 0, -1);

    % 1. Punktierte Fläche 
    \draw[pattern=dots, pattern color=blue!50!black, dashed, draw=blue!60!black] 
        (S3) -- (S2) -- (S1) -- (S4) -- cycle;

    % 1. Schraffierte innere Fläche ZUERST zeichnen (transparent ohne Hintergrundfüllung)
    \draw[pattern=north west lines, pattern color=blue!50!black, thick, draw=blue!60!black] 
        (T1) -- (T2) -- (T3) -- (T4) -- cycle;

    % 2. Gestrichelte Kanten nach Vorgabe: (AB), (AD), (AE)
    \draw[dashed, gray, thick] (A) -- (B);
    \draw[dashed, gray, thick] (A) -- (D);
    \draw[dashed, gray, thick] (A) -- (E);

    % 3. Durchgezogene Kanten nach Vorgabe: (BC), (BF), (CG), (CD), (DH), (EF), (EH), (FG), (GH)
    \draw[thick] (B) -- (C);
    \draw[thick] (B) -- (F);
    \draw[thick] (C) -- (G);
    \draw[thick] (C) -- (D);
    \draw[thick] (D) -- (H);
    \draw[thick] (E) -- (F);
    \draw[thick] (E) -- (H);
    \draw[thick] (F) -- (G);
    \draw[thick] (G) -- (H);
   % \draw[dashed,gray,very thick] (S3) -- (S4);

    % 4. Koordinatenachsen ab Ursprung (0,0,0)
    \draw[->, thick, red!70!black] (0,0,0) -- (2,0,0) node[anchor=north east] {$x$};
    \draw[->, thick, red!70!black] (0,0,0) -- (0,1.8,0) node[anchor=south west] {$y$};
    \draw[->, thick, red!70!black] (0,0,0) -- (0,0,1.4) node[anchor=south east] {$z$};

    % Ursprung markieren
    \fill[red!70!black] (0,0,0) circle (1pt);
    \node[red!70!black] at (0,0,0) [anchor=north] {$O$};

    % 5. Beschriftung der Eckpunkte des Würfels
%    \node[font=\footnotesize, below left] at (A) {$A$};
%    \node[font=\footnotesize, below right] at (B) {$B$};
%    \node[font=\footnotesize, above right] at (C) {$C$};
%    \node[font=\footnotesize, above left] at (D) {$D$};
%    \node[font=\footnotesize, left] at (E) {$E$};
%    \node[font=\footnotesize, right] at (F) {$F$};
%    \node[font=\footnotesize, above right] at (G) {$G$};
%    \node[font=\footnotesize, above left] at (H) {$H$};

    % Beschriftung der schraffierten Fläche (mit Führungslinie)
%    \draw[thin, blue!60!black] (0, 0.2, 0.5) -- (-0.8, 0.4, 0.5) 
%        node[right, align=left, font=\small] {$D^\parallel =\{0\} \times [-\frac{1}{2},\frac{1}{2}] \times [0,1)$};

    % Titel
 %   \node[above] at (0, 0, 1.6) {\textbf{Würfel $[-1,1]^3$ (mit verkürzter $x$-Achse)}};

\end{tikzpicture}
\caption{The figure shows the  half-cube $[-1,0]\times [-1,1]^2$ and  the surfaces $\Sigma_1$ (dotted) 
and $\{0\}\times {]0,1[}\times {]-1,1[}$ (hatched).
The boundary condition is Dirichlet outside the hatched surface.
On the latter, the Neumann condition arises from symmetric reflection.
}
\label{fig:4.11.2}
\end{figure}

%\begin{figure}[h]
%    \centerline{\includegraphics[scale=0.9]{BuriedBild5.pdf}}
%     \caption{${\bf  Q}_-$ with the dotted Neumann part resulting from symmetric reflection. The rest is Dirichlet 
%boundary, the black part resulting from $D^\parallel$. }
%\end{figure}
%
\item If $M =\bigl ( ]-1,1[ \times ]-1,1[ \times \{0\}\bigr ) \setminus \Sigma_1$:
\begin{equation} \label{e-possi3}
\Lambda_-= {\bf Q}_-, \; M_-= ({ ]-1,0[} \times {]-1,1[} \times \{0\}) \cup (\{0\} \times {]0,1[} \times {]-1,0[}) ,
\end{equation}
\item If $ M = ]-1,1[ \times ]0,1[ \times \{0\}$:
\begin{equation} \label{e-possi4}
\Lambda_-= {\bf Q}_-, \quad M_-= (]-1,0[ \times ]0,1[ \times \{0\}) \cup (\{0\} \times ]0,1[ \times ]-1,0])
\end{equation}
\end{itemize} 
All these settings are included in our regularity results: \eqref{e-possi1} and   
\eqref{e-possi2} are covered by Lemma \ref{c-modifymodel} ii).
\eqref{e-possi3} is treated in Corollary \ref{c-Neumannkant} and \eqref{e-possi4} is treated in
Lemma \ref{c-modifymodel} i). Thus, in any case, $ v|_{\Lambda_-}$ admits the estimate 
\[
\| v|_{\Lambda_-}\|_{W^{1,q}_{E_-}(\Lambda_-)} \le c \|g|_{\Lambda_-}\|_{W^{-1,q}_{E_-}(\Lambda_-)}
\le  c \|g \|_{W_E^{-1,q}(\Lambda)}.
\]
Since $v \in W^{1,2}_E(\Lambda)$ is invariant under $ \Psi_2$,  $\| v \|_{W^{1,q}_{E}(\Lambda)}
=2\| v|_{\Lambda_-}\|_{W^{1,q}_{E_-}(\Lambda_-)}$, and we obtain the estimate
\begin{equation} \label{e-absdch}
\| v\|_{W^{1,q}_E(\Lambda)} \le  c \|g \|_{W_E^{-1,q}(\Lambda)}
\end{equation}
for a suitable constant $c$ and $q \in [2,p[$,  if $p>3$ is sufficiently close to $3$. 
The operator $ \Psi_{2}^*$ transforms $W_E^{-1,q}(\Lambda)$ continuously into itself, 
and $L^2(\Lambda)$ is dense in $W_E^{-1,q}(\Lambda)$, so
\eqref{e-absdch} implies $u + \Psi_2 u \in W_E^{1,q}(\Lambda)$ \emph{for all} $f \in W_E^{-1,q}(\Lambda)$. 
Together with the discussion of \eqref{e-gleich1} this yields
$u \in W^{1,q}_E(\Lambda)$, if $f \in W_E^{-1,q}(\Lambda)$. Thus, the assertion 
is obtained from the open mapping theorem.
\end{proof}

It follows the final step of the proof of Theorem \ref{t-main}: Since the cases  $\x \in D^\parallel_d$ and 
$\x \in R_d$ are already established, the remaining ones are $\x \in D^\parallel_c$ and $\x \in R_c$.
With a slight change of notation define $U_{\x}$ by: 
\begin{equation} \label{e-defUx}
U_\x = \begin{cases} 
\widetilde U_\x \; \text{as in Lemma \ref{c-transform2}},  &\x \in D^\parallel_c\\
U_\x\; \text {as in Assumption } \ref{d-Rout},&\x \in R_c.
\end{cases}
\end{equation}
 We next apply Lemma \ref{l-cutoff} with $V= U_\x$. So, for $\eta \in C^\infty_0(U_\x)$, one gets
 $\eta u \in W^{1,2}_{D_\x}(\Omega_\x)$ with $\Omega_\x :=\Omega \cap  U_\x$
and 
\[
D_\x= \begin{cases}
\partial \Omega_\x, & \x \in D^\parallel_c  \\
\partial \Omega_\x \setminus \bigl ( U_\x \cap (\partial \Omega \setminus D)
\bigr ),
&\x \in R_c.
\end{cases}
\]
 Moreover,  Lemma \ref{l-cutoff} tells us that the function
 $\eta u|_{\Omega \cap U_\x}=:u_\bullet$ satisfies 
\[
-\nabla \cdot \mu_\bullet \nabla u_\bullet = f_\bullet \in W_{D_\x}^{-1,q}(\Omega_\x),
\quad \mu_\bullet := \mu|_{\Omega_\x}
\]
with $\| f_\bullet \|_{W_{D_\x}^{-1,q}(\Omega \cap  U_\x)} \le c 
\|f\|_{W^{1,q}_D(\Omega)}$, $c$ independent of $f$.\\
After passing to this localized problem, we employ now the transformation principle, Proposition \ref{p-transform}.
Then we are almost  in the situation of Lemma \ref{l-multiple1}: The geometries one has to treat are 
exactly those, namely $\Lambda = {\bf C}\setminus \Sigma_1$ or $\Lambda = {\bf Q} \setminus \Sigma_1$. 
The difference is that the occurring coefficient functions are not constant on the corresponding subsets
up to now. But they have limits for $\z \to 0$: For $\x \in D^\parallel_c$ and $\Lambda = {\bf C}\setminus \Sigma_1$
 this limit exists, due to Proposition 
\ref{p-transform} and Lemma \ref{c-transform2}, in $\Lambda \cap {\bf  H}^\pm_3$.
%Modifying the coefficient function by taking it to be the corresponding limit, which is indeed constant  on
% $\Lambda \cap {\bf H}^\pm_3$,
So  Lemma \ref{l-multiple1} in combination with Lemma \ref{l-localize} provides the asserted  $W^{1,q}$ regularity. 
For $\x \in R_c$ and $\Lambda = {\bf Q } \setminus \Sigma_1$ the argument is the same.

\section[Concluding remarks]
{Concluding remarks}\label{s-conclremarks} 
\renewcommand{\labelenumi}{\roman{enumi})}
\begin{enumerate}
\item \label{r-stoerung}
The condition \eqref{e-transiota} in Assumption  \ref{d-Rout} ii) can be perturbed slightly. This means that one
can add to a coefficient function, fulfilling this condition, another one which is sufficiently small in the $L^\infty$-norm,
and the main result (Theorem \ref{t-main}) remains true. This follows by straight forward perturbation arguments,  
resting on \eqref{e-opnormgegenLinftiy}.\\
We have tried hard, resting on this argument and gauging, to avoid  Assumption \ref{d-Rout} ii)  at all, since it is really a severe restriction concerning the admissable configurations. Unfortunately, all these attempts have failed.

%\item \label{r-away}
%It is not by accident that we had only points from $\overline {D^\parallel}$ under consideration here.
% If other boundary points are involved, then one can apply the (local) regularity results obtained in \cite{disser}.
% Moreover, Theorem \ref{t-groeger_local_global} then allows to deduce a \emph{global} regularity result
% within the $W^{1,q}$-scale.

\item \label{r-parabolic}
The global $W^{1,q}$ regularity result opens the door to the investigation of related elliptic/parabolic 
systems as, e.g., the van Roosbroeck system, describing the flow of electrons and holes in a semiconductor. 
These, reformulated to a quasilinear parabolic system as in \cite{Kaiser}, can be resolved since the involved elliptic operators admit, besides 
$W^{1,q}$ regularity, maximal parabolic regularity. The latter is due to fact that second order elliptic operators 
with real coefficients generate contractive $C_0$ semigroups on $L^p$ spaces (see \cite[Ch. 4.3.1]{Ouhabaz}) and the famous theorem of Lamberton \cite{Lamberton} applies.

\item \label{r-2d}
Quite similar results are obtained in two space dimensions, this time under much more general conditions.
Namely, if $D\subset \overline \Omega$ is a closed $1$-set and there is a continuous extension operator
$\mathfrak E:W^{1,q}_D(\Omega) \to  W^{1,q}_D(\R^3)$, then
\[
-\nabla \cdot \mu \nabla +1:W^{1,q}_D(\Omega) \to :W^{-1,q}_D(\Omega)
\]
is a topological isomorphism for $q \in {]2-\delta, 2 + \delta[}$ and some $\delta >0$, see \cite{jons}.

\end{enumerate}

{\bf Acknowledgment} We thank Moritz Egert (Darmstadt) for many helpful discussions and the referees for their careful reading and useful suggestions that improved the exposition of the paper.  
\\

\providecommand{\bysame}{\leavevmode\hbox to3em{\hrulefill}\thinspace}

\end{document}